\DeclareSymbolFont{AMSb}{U}{msb}{m}{n}
\documentclass[11pt,a4paper,twoside,leqno,noamsfonts]{amsart}

\DeclareMathAlphabet{\mathbbm}{U}{bbm}{m}{n}
\usepackage[dvipsnames]{xcolor}
\definecolor{auburn}{rgb}{0.43, 0.21, 0.1}
\usepackage[english]{babel}
\usepackage{indentfirst}
\usepackage[foot]{amsaddr}
\usepackage[utopia]{mathdesign}
    \DeclareSymbolFont{usualmathcal}{OMS}{cmsy}{m}{n}
    \DeclareSymbolFontAlphabet{\mathcal}{usualmathcal}
\usepackage[a4paper,top=3.5cm,bottom=3cm,left=3cm,right=3cm,bindingoffset=5mm]{geometry}
\usepackage{braket,caption,comment}
\usepackage{multirow,booktabs,microtype}
\usepackage[colorlinks,bookmarks]{hyperref} 
   \hypersetup{colorlinks,%
   urlcolor=.,%
   citecolor=auburn,%
   filecolor=black,%
   linkcolor=black}%
\numberwithin{equation}{section}
\renewenvironment{proof}{{\scshape Proof.}}{\qed\medskip}

\def\be{\begin{equation}}    
\def\ee{\end{equation}}
\def\bitem{\begin{itemize}}
\def\eitem{\end{itemize}}
\def\WL{\mathsf{WL}}
\def\benum{\begin{enumerate}}
\def\bw{\bigwedge}

\def\wt{\mathsf{wt}}

\def\eenum{\end{enumerate}}
\def\sra{\rightarrow}
\def\Scal{{\mathcal S}}
\def\nC{{\widetilde{C}}}

\def\Hcal{\mathcal H}

\def\Fcal{\mathcal F}
\def\rds{{\omega_\pi}}

\def\CC{\mathbb C}

\def\Ccal{{\mathcal C}}

\def\ra{\rightarrow}

\def\surj{\twoheadrightarrow}

\def\AA{\mathbb{A}}
\def\EE{\mathbb E}

\def\PP{\mathbb P}
\def\TS{\mathfrak X}

\def\C{\mathbb C}
\def\P{\mathbb P}
\def\Q{\mathbb Q}

\def\wrp{{{\mathbb W}_\pi}}

\def\N{\mathbb N}

\def\O{\mathscr O}

\DeclareMathOperator{\Pic}{Pic}

\DeclareMathOperator{\ord}{ord}

\DeclareMathOperator{\Spec}{Spec\,}

\DeclareMathOperator{\GL}{GL}

\DeclareMathOperator{\dd}{d}

\DeclareMathOperator{\Sym}{Sym}

\DeclareMathOperator{\Ext}{Ext}

\DeclareMathOperator{\AD}{AD}

\newenvironment{citazione}
  {\begin{quotation}}
  {\end{quotation}}

\makeatletter
\newenvironment{proofof}[1]{\par
  \pushQED{\qed}%
  \normalfont \topsep6\p@\@plus6\p@\relax
  \trivlist
  \item[\hskip\labelsep
        \scshape
    Proof of #1\@addpunct{.}]\ignorespaces
}{%
  \popQED\endtrivlist\@endpefalse
}
\makeatother

\newtheoremstyle{conv} 
        {4mm}
        {4mm}
        {\rmfamily}
        {4mm}
        {\itshape}
        {.}
        {1mm}
        {}
\theoremstyle{conv}

\newtheorem*{notation*}{Notation}
\newtheorem*{conventions}{Conventions}

\newtheoremstyle{thm} 
        {4mm}
        {4mm}
        {\slshape}
        {4mm}
        {\scshape}
        {.}
        {1mm}
        {}
\theoremstyle{thm}
\newtheorem{prop}{Proposition}[section]
\newtheorem*{teo*}{Theorem}
\newtheorem{lemma}[prop]{Lemma}

\newtheorem{teo}{Theorem}[section]

\theoremstyle{definition}
\newtheorem{example}[prop]{Example}
\newtheorem{definition}[prop]{Definition}
\newtheorem{remark}[prop]{Remark}

\usepackage{tikz}

\usepackage{tikz-cd}
\usepackage{rotating}

\usetikzlibrary{matrix,shapes,arrows,decorations.pathmorphing}
\tikzset{commutative diagrams/arrow style=math font}
\tikzset{commutative diagrams/.cd,
mysymbol/.style={start anchor=center,end anchor=center,draw=none}}

\newcommand*{\defeq}{\mathrel{\vcenter{\baselineskip0.5ex \lineskiplimit0pt
                     \hbox{\scriptsize.}\hbox{\scriptsize.}}}%
                     =}

\title{Jet Bundles on Gorenstein Curves and Applications}

\author[L. Gatto]{Letterio Gatto}
\address{Politecnico di Torino}
\email{letterio.gatto@polito.it}

\author[A. T. Ricolfi]{Andrea T. Ricolfi}
\address{Max-Planck-Institut f\"{u}r Mathematik}
\email{atricolfi@mpim-bonn.mpg.de}

\thanks{The first author was partially supported by INDAM-GNSAGA and by PRIN ``Geometria sulle variet\`a algebriche''. The visit of the  first author to Stavanger and of the second author to Torino was supported by the grant ``Finanziamento Diffuso della Ricerca'' by Politecnico di Torino}

\begin{document}

\dedicatory{Dedicated to Professor Goo Ishikawa, on the occasion of his 60th birthday}

\begin{abstract}
In the last twenty years a number of papers appeared aiming to construct locally free replacements of the sheaf of principal parts for families of Gorenstein curves.
The main goal of this survey is to present to the widest possible mathematical audience a catalogue of such constructions, discussing the related literature and reporting on a few  applications to classical problems in Enumerative Algebraic Geometry. 
\end{abstract}

\maketitle
\tableofcontents

\section{Introduction}
The purpose of this expository paper is to present a catalogue of locally free replacements of the sheaves of {\em principal parts} for (families of) Gorenstein curves. In the smooth category, locally free sheaves of principal parts  are better known as {\em jet bundles}, understood as those locally free sheaves whose transition functions reflect the transformation rules of the partial derivatives of a local section under a change of local coordinates (more details in Section~\ref{sec78798}). 
Being a natural globalisation of the fundamental notion of Taylor expansion of a function in a neighborhood of a point, jet bundles are ubiquitous in Mathematics. They proved powerful tools for the study of deformation theories within a wide variety of mathematical situations and have a number of purely algebraic incarnations: besides the aforementioned \emph{principal parts} of a quasi-coherent sheaf \cite{MR0238860} we should mention, for instance, the theory of {\em arc spaces} on algebraic varieties \cite{DenefLoeser1,LooijengaMM}, introduced by Nash in \cite{Nash} to deal with resolutions of singular loci of singular varieties.

The issue we want to cope with in this survey is that sheaves of principal parts of vector bundles defined on a singular variety $X$ are not locally free. Roughly speaking, the reason is that the analytic construction carried out in the smooth category, based on gluing local expressions of sections together with their partial derivatives, up to a given order, is no longer available. Indeed, around singular points there are no local parameters with respect to which one can take derivatives. This is yet another way of saying that the sheaf $\Omega^1_X$ of sections of the cotangent bundle is not locally free at the singular points. 

If $C$ is a projective reduced singular curve, it is desirable, in many interesting situations, to dispose of a notion of global derivative  of a regular section. If the singularities of $C$ are mild, that is, if they are \emph{Gorenstein}, locally free substitutes of the classical principal parts can be constructed by exploiting a natural derivation $\O_C\ra \omega_C$, taking values in the dualising sheaf, which by the Gorenstein condition is an \emph{invertible} sheaf. This allows one to mimic the usual procedure adopted in the smooth category. Related constructions have recently been reconsidered by A.~Patel and A.~Swaminathan in \cite{InvParts}, under the name of {\em sheaves of invincible parts}, motivated by the classical problem of counting hyperflexes in one-parameter families of plane curves. Besides \emph{loc.~cit.}, locally free jets on Gorenstein curves have been investigated by a number of authors, starting about twenty years ago \cite{LaksovThorup1,LaksovThorup3,LaksovThorup4,Esteves2, Gatto2}. The reader can consult \cite{Esteves1,Gatto1,GattoEsteves}, and the references therein, for several applications.

\subsection{The role of jet bundles in Algebraic Geometry}

The importance of jet extensions of line bundles in algebraic geometry emerges from their ability to  provide the proper flexible framework where to formulate and solve elementary but classical enumerative questions, such as: 
\bitem
\item [(i)] How many flexes does a plane curve possess?
\item [(ii)] How many members in a generic pencil of plane curves have a hyperflex?
\item [(iii)] How many fibres in a one-parameter family of curves of genus $3$ are hyperelliptic?
\item [(iv)] What is the class, in the rational Picard group of $\overline{M}_g$,  the moduli space of stable curves of genus $g$, of the closure of the locus of smooth curves possessing a special Weierstrass point?
\eitem

\medskip
We will touch upon each of these problems in this survey report.

\subsection{Wronskian sections over Gorenstein curves} 
A theory of {\em ramification points of linear systems on Gorenstein curves} was proposed in 1984 by C.~Widland in his Ph.D. thesis, also exposed in a number of joint papers with Robert F.~Lax \cite{WidLax1, WidLax2}. The dualising sheaf $\omega_C$ on an integral curve $C$, first defined by Rosenlicht \cite{Rosenlicht} via residues on the normalisation $\widetilde C$, can be realised as the sheaf of regular differentials on $C$, as explained by Serre in \cite[Ch.~4 \S~3]{serre2}.
There is a natural map $\Omega^1_C\sra \omega_C$ allowing one to define a derivation $\dd\colon\O_C\sra \omega_C$, by composition with the universal derivation $\O_C\sra \Omega^1_C$. Differentiating local regular functions by means of this composed differential allowed  Widland \cite{WidlandTh} and Lax  to define a global \emph{Wronskian section} associated to a linear system on a Gorenstein curve $C$, coinciding with the classical one for smooth curves.

As a quick illustration of how  such construction works, consider a plane curve $\iota\colon C\hookrightarrow \P^2$ of degree $d$, carrying the degree $d$ line bundle $\O_C(1)=\iota^\ast\O_{\P^2}(1)$. The Wronskian  by Widland and Lax vanishes along all the flexes of $C$, but also at singular points. The total order of vanishing equals the number of flexes on a smooth curve of the same degree. For example, if $C$ is an irreducible nodal plane cubic, the Wronskian associated to the bundle $\O_C(1)$ would  vanish at three smooth flexes, but also at the node with multiplicity $6$. If $C$ were cuspidal, the Wronskian  would vanish at the unique smooth flex, and at the cusp with multiplicity $8$. In all cases the ``total number'' (which is $9$) of inflection points is conserved. 

In sum, the  Wronskian defined by Widland and Lax is able to recover the classical Pl\"ucker formula counting smooth flexes on singular curves, but within a framework that is particularly suited to deal with degeneration problems, provided one learns how to extend  it to families. For families of smooth curves, as pointed out by Laksov \cite{Laksov1},  the Wronskian section of a relative line bundle should be thought of as the determinant of a map from the pullback of the Hodge bundle to a jet bundle. The theory by Widland and Lax, however, was lacking a suitable notion of jet bundles for Gorenstein curves, as Ragni Piene \cite{Piene} remarked in her AMS review  of \cite{WidLax1}:

\begin{citazione}
``This (Widland and Lax) Wronskian is a section of the line bundle $L^{\otimes s}\otimes \omega_C^{\otimes (s-1)s/2}$ , where $s\defeq 
\dim H^0 (X, L)$. They define the section locally and show that it patches. (In the classical case in which $X$ is smooth, one easily defines the Wronskian globally, by using the
$(s-1)$st sheaf of principal parts on $X$ of $L$. To do this in the present case, one would need a generalisation of these sheaves, where $\omega$ plays the role of $\Omega^1_X$. Such a generalisation is known only for $s=2$.)''
\end{citazione}

These generalisations are nowadays available in the aforementioned references. In the last two sections we will present a few applications and open questions arising from the use of such an extended notion of jet bundles for one-parameter families of stable curves. 

\subsection{Overview of contents}

In the first section we describe the construction of principal parts, jet bundles (with a glimpse on an abstract construction by Laksov and Thorup) and {\em invincible parts} by Patel and Swaminathan. In Section \ref{sec:appl} we describe two applications of locally free replacements: the enumeration of hyperflexes in families of plane curves via automatic degeneracies \cite{InvParts}, and the determination of the class of the stable hyperelliptic locus in genus $3$ \cite{Esteves1}. In Section \ref{sec:ram} we define ramification points of linear systems on smooth curves; we introduce the classical Wronskian section attached to a linear system and state the associated Brill--Segre formula. In Section \ref{WL} we describe a generalisation to Gorenstein curves, due to Lax and Widland. In Section \ref{sec:specialWP} we review the main ingredients needed in the  computation of the class in $\Pic(\overline M_g)\otimes \Q$ of the locus of curves possessing a special Weierstrass point as in \cite{Gatto1}. In Section \ref{sec:Examples} we propose a few examples and some natural but still open questions.

\medskip
\begin{conventions}
All schemes are Noetherian and defined over $\C$. Any scheme $X$ comes equipped with a sheaf of $\C$-algebras $\O_X$. If $U\subset X$ is an open subset in the Zariski (resp.~analytic) topology, then $\O_X(U)$ is the ring of regular (resp.~holomorphic) functions on $U$. A \emph{curve} is a reduced, purely $1$-dimensional scheme of finite type over $\C$. We denote by $K_C$ the canonical line bundle of a smooth curve $C$. In the presence of singularities, we will write $\omega_C$ for the dualising sheaf. We denote by $\Omega^1_\pi$ the sheaf of relative K\"{a}hler differentials on a (flat) family of curves $\pi\colon X\ra S$.
\end{conventions}

\bigskip
{\bf Acknowledgment.} Both authors are grateful to the anonymous referee for carefully reading the paper and for providing valuable comments, that definitely improved the shape of the paper in terms of clarity and readability. The first author is also indebted to Professor Stanis\l aw (Staszek) Janeczko for encouraging support. The second author wishes to thank Max-Planck Institut f\"{u}r Mathematik for support.
\smallskip

{\em This paper is dedicated to Professor Goo Ishikawa, on the occasion of the celebration (Goo '60) of his sixtieth birthday, wishing him many more years of new beautiful theorems.}

\section{Principal parts, jets and invincible parts}\label{sec:sgnacchera}

This first section is devoted to recall the definition and properties of the sheaves of principal parts and to introduce a couple of related constructions: jets of vector bundles, especially those of rank $1$, and the Patel--Swaminathan invincible parts. We start by giving the general idea of jets, which blends their analytic construction with the algebraic presence of the dualising sheaf.

These constructions lead to the technique of locally free replacements of principal parts for families of curves with at worst Gorenstein singularities. They are intended  to deal with degenerations of ramification points of  linear systems in one parameter families of curves of fixed arithmetic genus. In fact, 
in Section \ref{sec:appl} we shall give two applications to see the theory in action: the count of hyperflexes in a pencil, as performed in \cite{InvParts}, and the determination of the class of the stable hyperelliptic locus in genus $3$, as worked out by Esteves \cite{Esteves1}.

\subsection{The idea of jets}

Our guiding idea is the following ansatz, which we shall implement below only in the case of algebraic curves. Let $X$ be a (not necessarily smooth) complex algebraic variety of dimension $r$. 
If $X$ is not smooth, the sheaf of differentials $\Omega^1_X$ is not locally free. Even in this case it is possible to construct, in a purely algebraic fashion, the sheaf of principal parts (see Section \ref{sec:pp}) attached to  any quasi-coherent sheaf $\mathscr M$. If $X$ is singular, this sheaf is not locally free (even if $\mathscr M$ is locally free), and this makes harder its use  even to solve elementary enumerative problems. But suppose one has an $\O_X$-module homomorphism $\phi\colon\Omega^1_X\ra \mathscr M$, where $\mathscr M$ is a locally free sheaf of rank $r=\dim X$. This induces a derivation $\dd\colon\O_X\sra \mathscr M$ obtained by composing $\phi$ with the universal derivation $\O_X\sra \Omega^1_X$ attached to $X$. 
Let $P\in X$ be a point and $U$ an open neighborhood of $P$ trivialising $\mathscr M$, that is,
$$
\mathscr M(U)=\O(U)\cdot m_1\oplus\cdots\oplus \O(U)\cdot m_r.
$$
Such a trivialisation allows one to define partial derivatives with respect to the generators $m_1,\ldots,m_r\in \mathscr M(U)$. In the smooth case, and taking $\mathscr M=\Omega^1_X$, these generators can just be taken to be the differentials of a local system of parameters around $P$.
Following an idea essentially due to Lax and Widland, one defines for each $f\in \O(U)$ its ``partial derivatives'' $d_if\in \O(U)$ by means of the relation
\[
\dd f = \sum_{i=1}^rd_if\cdot m_i
\] 
in $\mathscr M(U)$. Iterating this process in the obvious way, one can define higher order partial derivatives (with respect to $m_1,\ldots,m_r$), and thus jet bundles, precisely as in the smooth category.

\subsection{Dualising sheaves} 
This technical section can be skipped at a first reading. It will be applied below in special cases only, but it is important because it puts the subject in the perspective of new applications.

Any proper flat family of curves $\pi\colon X\ra S$ has a dualising complex $\omega_\pi^\cdot \defeq \pi^!\O_S$. Here $\pi^!$ is the right adjoint to $R\pi_\ast$. The cohomology sheaf of the dualising complex
\[
\omega_\pi = h^{-1}(\omega_\pi^\cdot),
\]
in degree $-1$ (where $1$ is the relative dimension of $\pi$) is called the \emph{relative dualising sheaf} of the family.
Its formation commutes with arbitrary base change; for instance, we have
\[
\omega_\pi\big|_{X_s} = \omega_{X_s}
\]
for $X_s = \pi^{-1}(s)$ a fibre of $\pi$.

\begin{example}\label{gnagna98}
Let $\pi\colon X\ra S$ be a local complete intersection morphism. This means that there is a factorisation $\pi\colon X\ra Y\ra S$ with $i\colon X\ra Y$ a regular immersion and $Y\ra S$ a smooth morphism. Then one can compute the dualising sheaf of $\pi$ as
\be\label{omegalci}
\omega_\pi = \det(\mathscr I/\mathscr I^2)^\vee\otimes_{\O_X}i^\ast\det \,\Omega^1_{Y/S},
\ee
where $\mathscr I\subset \O_Y$ is the ideal sheaf of $X$ in $Y$. 
Every curve in a smooth surface is a local complete intersection scheme. For instance, if $i\colon C\hookrightarrow \P^2$ is a plane curve of degree $d$,  the ideal sheaf of $i$ is $\O_{\P^2}(-d)$ and so \eqref{omegalci} yields
\[
\omega_{C} = \O_C(d)\otimes_{\O_C} i^\ast\det\,\Omega^1_{\P^2} = \O_C(d-3).
\]
\end{example}

\begin{definition}\label{def:gor323}
A (proper) $\C$-scheme $X$ is said to be \emph{Cohen--Macaulay} if its dualising complex $\omega_X^\cdot$ is quasi-isomorphic to a sheaf. When this sheaf, necessarily isomorphic to $\omega_X$, is invertible, $X$ is called \emph{Gorenstein}.
\end{definition}

For a proper flat morphism $\pi\colon X\ra S$, the relative dualising sheaf
$\omega_\pi$ is invertible precisely when $\pi$ has Gorenstein fibres.

\subsection{Principal parts}\label{sec:pp}    
Sheaves of principal parts were introduced in \cite[Ch.~16.3]{MR0238860}.
Let $\pi\colon X\ra S$ be a morphism of schemes, $\mathscr I$ the ideal sheaf of the diagonal $\Delta\colon X\ra X\times_SX$ and denote by $\Omega^1_\pi=\Delta^\ast(\mathscr I/\mathscr I^2)$ the sheaf of relative K\"{a}hler differentials. Let $p$ and $q$ denote the projections $X\times_SX\ra X$, and denote by $\Delta_k\subset X\times_SX$ the closed subscheme defined by $\mathscr I^{k+1}$, for every $k\geq 0$. Then, for every quasi-coherent $\O_X$-module $E$, the sheaf
\[
P^k_\pi(E) \defeq p_\ast\left(q^\ast E\otimes \O_{\Delta_k}\right)
\]
is quasi-coherent and is called the $k$-th \emph{sheaf of principal parts} associated to the pair $(\pi,E)$. When $S=\Spec \C$ we simply write $P^k(E)$ instead of $P^k_\pi(E)$.

\begin{prop}\label{prop:ses}
Let $\pi\colon X\ra S$ be a smooth morphism, $E$ a quasi-coherent $\O_X$-module. The sheaves of principal parts fit into right exact sequences
\[
E\otimes \Sym^k\Omega^1_\pi\ra P^k_\pi(E) \ra P^{k-1}_\pi(E) \ra 0
\]
for every $k\geq 1$. If $E$ is locally free then the sequence is exact on the left, and $P^k_\pi(E)$ is locally free for all $k\geq 0$.
\end{prop}

\begin{proof}
Consider the short exact sequence
\[
0\ra \mathscr I^k/\mathscr I^{k+1}\ra \O_{\Delta_k}\ra\O_{\Delta_{k-1}}\ra 0. 
\]
Tensoring it with $q^\ast E$ gives an exact sequence
\be\label{rightex}
q^\ast E\otimes \mathscr I^k/\mathscr I^{k+1}\xrightarrow{\epsilon} q^\ast E\otimes \O_{\Delta_k} \ra q^\ast E\otimes \O_{\Delta_{k-1}}\ra 0.
\ee
The sheaf $q^\ast E\otimes \mathscr I^k/\mathscr I^{k+1}$ is supported on the diagonal $\Delta_0\subset X\times_SX$, and the same is true for its quotient $\mathscr Q\defeq(q^\ast E\otimes \mathscr I^k/\mathscr I^{k+1})/\ker\epsilon \subset q^\ast E\otimes \O_{\Delta_k}$. Since $p|_{\Delta_0}$ is an isomorphism, we have $R^ip_\ast \mathscr F = 0$ for all $i>0$ and all sheaves $\mathscr F$ supported on $\Delta_0$.
Therefore, applying $p_\ast$ to \eqref{rightex} we obtain 
\be\label{res98}
p_\ast \left(q^\ast E\otimes \mathscr I^k/\mathscr I^{k+1}\right)\ra P^k_\pi(E)\ra P^{k-1}_\pi(E) \ra R^1p_\ast\mathscr Q=0,
\ee
which is the required exact sequence, since
\begin{align*}
p_\ast \left(q^\ast E\otimes \mathscr I^k/\mathscr I^{k+1}\right)
&=\Delta^\ast \left(q^\ast E\otimes \mathscr I^k/\mathscr I^{k+1}\right)\\
&=\Delta^\ast q^\ast E\otimes \Delta^\ast \left(\mathscr I^k/\mathscr I^{k+1}\right)\\
&=E\otimes \Delta^\ast \Sym^k\left(\mathscr I/\mathscr I^2\right) \\
&=E\otimes \Sym^k\Omega^1_\pi.
\end{align*}
We used smoothness of $\pi$ to ensure that $\mathscr I$ is locally generated by a regular sequence. This allowed us to make the identification $\mathscr I^k/\mathscr I^{k+1} = \Sym^k(\mathscr I/\mathscr I^2)$ in the third equality above.
If $E$ is locally free, then \eqref{rightex} is exact on the left, and the same is true for \eqref{res98}, so that local freeness of $P^k_\pi(E)$ follows by induction exploiting the resulting short exact sequence and the base case provided by $P^0_\pi(E) = E$.
\end{proof}

\begin{example}
Suppose $\pi\colon X\ra S$ is smooth. Then there is a splitting $P^1_\pi(\O_X) = \O_X \oplus \Omega^1_\pi$. For an arbitrary vector bundle $E$, the splitting of the first bundle of principal parts usually fails even when $S$ is a point. In fact, in this case, the splitting is equivalent to the vanishing of the \emph{Atiyah class} of $E$, which by definition is the extension class
\[
A(E)\in \Ext^1_X(E,E\otimes \Omega_X^1)
\]
attached to the short exact sequence of Proposition \ref{prop:ses} taken with $k=1$. But the vanishing of the Atiyah class is known to be equivalent to the existence of a holomorphic connection on $E$.
\end{example}

Note that for every quasi-coherent sheaf $E$ on $X$ one has a canonical map
\be\label{map:82929}
\nu\colon \pi^\ast\pi_\ast E \ra p_\ast q^\ast E \ra P^k_\pi(E),
\ee
where the first one is an isomorphism when $\pi$ is flat, and the second one comes from applying $p_\ast(q^\ast E\otimes -)$ to the surjection $\O\surj \O_{\Delta_k}$.

\begin{example}
To illustrate the classical way of dealing with bundles of principal parts, we now compute the number $\delta$ of singular fibres in a general pencil of hypersurfaces of degree $d$ in $\P^n$. This calculation will be used in Subsection \ref{counthyperflex}.
The number $\delta$ is nothing but the degree of the discriminant hypersurface in the space of degree $d$ forms on $\P^n$, which in turn is the degree of
\[
c_n(P^1(\O_{\P^n}(d)))\in A^n(\P^n).
\]
By Proposition \ref{prop:ses}, the bundle $P^1(\O_{\P^n}(d))$ is an extension of $\O_{\P^n}(d)$ by $\Omega^1_{\P^n}(d)$. The Euler sequence
\[
0\ra \Omega^1_{\P^n}\ra \O_{\P^n}(-1)^{n+1}\ra \O_{\P^n}\ra 0
\]
twisted by $\O_{\P^n}(d)$ says that the same is true for the bundle $\O_{\P^n}(d-1)^{n+1}$. Then the Whitney sum formula implies that
\[
c(P^1(\O_{\P^n}(d))) = c(\O_{\P^n}(d-1)^{n+1}) = (1+(d-1)\zeta)^{n+1},
\]
where $\zeta\in A^1(\P^n)$ is the hyperplane class. Computing the $n$-th Chern class gives
\be\label{numberofnodes}
\delta = (n+1)\cdot (d-1)^n.
\ee
\end{example}

\subsection{Jet bundles}\label{sec78798}

Let $\pi\colon X\ra S$ be a quasi-projective local complete intersection morphism of constant relative dimension $d\geq 0$. Let $\Omega^1_\pi$ be the sheaf of relative differentials, and $\Omega^d_\pi$ its $d$-th exterior power. Then there exists a canonical morphism $\Omega^d_\pi\ra \omega_\pi$ restricting to the identity over the smooth locus of $\pi$ (see Corollary 4.13 in \cite[Section 6.4]{MR1917232} for a proof).
The construction goes as follows. Let $X\ra Y\ra S$ be a factorisation of $\pi$, with $i:X\ra Y$ a regular immersion with ideal $\mathscr I\subset \O_Y$ and $Y\ra S$ smooth. The exact sequence
\[
\mathscr I/\mathscr I^2\ra i^\ast\Omega^1_{Y/S}\ra \Omega^1_\pi\ra 0
\]
induces a canonical map
\[
\mu_Y\colon \Omega_\pi^d\otimes \det\,\mathscr I/\mathscr I^2 \ra i^\ast\det\,\Omega^1_{Y/S}.
\]
According to \eqref{omegalci}, tensoring $\mu_Y$ with the dual of $\det\,\mathscr I/\mathscr I^2$ gives a morphism $\Omega_\pi^d\ra \omega_\pi$. It is not difficult to see that this map does not depend on the choice of factorisation.

\medskip
A natural morphism of sheaves 
$\Omega^1_\pi \ra \omega_\pi$, restricting to the identity on the smooth locus of $\pi$, exists for arbitrary flat families $\pi\colon (X,x_0)\ra (S,0)$ of germs of reduced curves \cite[Prop.~4.2.1]{MR571575}. 
More generally, the results in \cite[Sec.~4.4]{MR518299} show that a natural morphism \be\label{gnocca}
\phi\colon \Omega^d_\pi \ra \omega_\pi,
\ee
can be constructed for every flat morphism $\pi\colon X\ra S$ of relative dimension $d$ over a reduced base $S$ (and over a field of characteristic zero). 

We now apply this construction to flat families $\pi\colon X\ra S$ of \emph{Gorenstein} curves (so for $d = 1$), taking advantage of the invertibility of $\omega_\pi$ in order to construct locally free jets. When dealing with such families, we will therefore assume to be working over a reduced base, which will be enough for all our applications.
Composing $\phi$ with the exterior derivative homomorphism $\dd\colon \O_X\ra \Omega^1_\pi$ attached to the family gives an $\O_S$-linear derivation
\be\label{der:dpi}
\textrm d_\pi\colon \O_X\ra \omega_\pi.
\ee
For every integer $k\geq 0$ and line bundle $L$ on $X$, there exists a vector bundle
\be\label{JetL}
J^k_\pi(L)
\ee
of rank $k+1$ on $X$, called the $k$-\emph{th jet extension} of $L$ relative to the family $\pi$. We refer to \cite[Section 2]{Gatto1} for its detailed construction in the case of stable curves. The same construction (as well as the proof of Proposition \ref{prop:exactsequence} below) extends to any family of Gorenstein curves as one only uses the map $\Omega^1_\pi\sra \omega_\pi$ and the invertibility of the relative dualising sheaf. The bundle \eqref{JetL} depends on the derivation $\textrm d_\pi$ (although we do not emphasise it in the notation), and formalises the idea of taking derivatives (with respect to $\textrm d_\pi$) of sections of $L$ along the fibres of $\pi$. It can be thought of as a holomorphic, or algebraic, analogue of the $\mathcal C^\infty$ bundle of coefficients of the Taylor expansion of the smooth functions on a differentiable manifold.
When $S = \Spec \C$ we simply write $J^k(L)$.

We now sketch the construction of the jet bundle \eqref{JetL}. Suppose we have an open covering $\mathcal U = \set{U_\alpha}$ of $X$, trivialising $\omega_\pi$ and $L$ at the same time, with generators $\epsilon_\alpha\in\omega_\pi(U_\alpha)$ and $\psi_\alpha\in L(U_\alpha)$ respectively over the ring of functions on $U_\alpha$.
Then for every non constant global section $\lambda\in H^0(X,L)$ we can write
\[
\lambda|_{U_\alpha} = \rho_\alpha\cdot \psi_\alpha\in L(U_\alpha)
\]
for certain functions $\rho_\alpha\in \O_X(U_\alpha)$.
Then one can define operators $D_\alpha^i\colon \O_{U_\alpha}\ra \O_{U_\alpha}$ inductively for $i\geq 0$, by letting $D_\alpha^0(\rho_\alpha) = \rho_\alpha$ and by the relation
\[
\textrm d_\pi(D_\alpha^{i-1}(\rho_\alpha)) = D_\alpha^i(\rho_\alpha)
\cdot \epsilon_\alpha.
\]
It is then an easy technical step to show that over the intersection $U_{\alpha\beta} = U_\alpha\cap U_\beta$, the $(k+1)$-vectors $(D^i_\alpha(\rho_\alpha))^T$ and $(D^i_\beta(\rho_\beta))^T$ differ by a matrix $M_{\alpha\beta}\in \GL_{k+1}(\O_{U_{\alpha\beta}})$, and that in fact the data $\{M_{\alpha\beta}\}$ define a $1$-cocycle with respect to $\mathcal U$. The verification of this fact uses that $\textrm d_\pi$ is a derivation. The upshot is that the vectors $(D^i_\alpha(\rho_\alpha))$ glue to a global section 
\be\label{troiazione}
D^k\lambda
\ee
of a well defined vector bundle $J^k_\pi(L)$.
Moreover, the bundle obtained comes with a natural $\C$-linear morphism
\be\label{deltamap}
\delta\colon \O_X\ra J^k_\pi(L)
\ee
such that if $J^k_\pi(L)|_{U_\alpha}$ is free with basis $\set{\epsilon_{\alpha,i}:0\leq i\leq k}$, then $\delta$ is
defined on this open patch by $f\mapsto \sum_{i=0}^kD_\alpha^i(f)\cdot \epsilon_{\alpha,i}$.

\begin{example}
When $S$ is a point, $X$ is a smooth projective curve, $L$ is the cotangent bundle $\Omega^1_X$ with the exterior derivative $\dd\colon \O_X\ra \Omega^1_X$, the $\C$-linear map \eqref{deltamap} reduces to the ``Taylor expansion'' truncated at order $k$. More precisely, let $U\subset X$ be an open subset (trivialising $\omega_X=\Omega^1_X$) with local coordinate $x$. Then we can take $\epsilon=\dd x\in \Omega^1_X(U)$ as an $\O_X(U)$-linear generator, and $\set{\dd x^i:0\leq i\leq k}$ can be taken as a basis of $J^k(\Omega^1_X)|_U$. The restriction $\delta|_U$ of \eqref{deltamap} then takes the form
\[
f\mapsto \sum_{i=0}^k \frac{1}{i!}\frac{\partial^if}{\partial x^i}\dd x^i,
\]
where the denominator $1/i!$ is there for cosmetic reasons. The cocycle condition that the above coefficients need to satisfy is equivalent to the \emph{chain rule} for holomorphic functions.
\end{example}

Computations in intersection theory involving jet bundles often rely on the application of the following key result. 
\begin{prop}[{\cite[Prop.~2.5]{Gatto1}}] \label{prop:exactsequence}
Let $\pi\colon X\ra S$ be a flat family of Gorenstein curves. Then, for every $k\geq 1$ and line bundle $L$ on $X$, there is an exact sequence of vector bundles
\be\label{sesjets1}
0\ra L\otimes \omega_\pi^{\otimes k}\ra J^k_\pi(L)\ra J^{k-1}_\pi(L)\ra 0.
\ee
\end{prop}

\begin{lemma}\label{lemma:smoothlocus}
Let $\pi\colon X\ra S$ be a flat family of Gorenstein curves with smooth locus $U\subset X$, let $L$ be a line bundle on $X$, and fix an integer $k\geq 0$. Then 
\[
J^k_\pi(L)\big|_U = P^k_\pi(L)\big|_U.
\]
\end{lemma}

\begin{proof}
The derivation $\dd_\pi\colon \O_X\ra \omega_\pi$ defined in \eqref{der:dpi} and used to define the $k$-jets restricts to the universal derivation $\dd\colon \O_U\ra \Omega^1_{U/S}$ over the smooth locus $U$. But jet bundles taken with respect to the universal derivation agree with principal parts in the smooth case, as one can verify directly from their construction; see also \cite[Section 4.11]{LaksovThorup1} for a reference.
\end{proof}

\subsubsection{The approach of Laksov and Thorup}
Laksov and Thorup \cite{LaksovThorup1} generalised the construction of \eqref{deltamap} in the following sense.
Given an $S$-scheme $X$ and a quasi-coherent $\O_X$-module $\mathscr M$ admitting an $\O_S$-linear derivation $\dd\colon \O_X\ra \mathscr M$, they constructed for all $k\geq 0$ an $\O_S$-algebra 
\[
\mathcal J^k = \mathcal J^k_{\mathscr M,\dd}
\]
over $X$, along with an algebra map $\delta\colon\O_X\ra\mathcal J^k$ generalising the one constructed in \eqref{deltamap}. 
The sheaf $\mathcal J^k$ is called the $k$-th \emph{algebra of jets}. It is quasi-coherent, and of finite type whenever $\mathscr M$ is. For every $\O_X$-module $\mathscr L$, one can consider the $\O_X$-module 
\[
\mathcal J^k(\mathscr L) = \mathcal J^k\otimes_{\O_X}\mathscr L
\]
of $\mathscr L$-\emph{twisted jets}. They fit into exact sequences
\[
\mathscr L\otimes\mathscr M^{\otimes k} \ra 
\mathcal J^k(\mathscr L) \ra
\mathcal J^{k-1}(\mathscr L) \ra
0,
\]
that are left exact whenever $\mathscr M$ is $S$-flat. The construction carried out in \cite{LaksovThorup1} works over fields of arbitrary characteristic and is completely intrinsic, in particular it avoids the technical step of verifying the cocycle condition.

\subsubsection{Arc spaces}
The study of arc spaces (also called jet schemes) was initiated by Nash \cite{Nash} in the $60$'s in the context of Singularity Theory. Arcs on algebraic varieties received a lot of attention more recently since Kontsevich's lecture \cite{KontsevichOrsay}. See for instance the papers by Denef--Loeser \cite{DenefLoeser1,DeLo2} and Looijenga \cite{LooijengaMM}. An arc of order $n$ on a variety $X$ based at point $P$ is a morphism
\[
\alpha\colon \Spec \C[t]/t^{n+1}\ra X
\]
sending the closed point to $P$. 
The reader may correctly think of it as the expression of a germ of complex curve considered together with its first $n$ derivatives. For instance if $n=1$, one obtains the classical notion of tangent space at a point.  
These maps form an algebraic variety $\mathcal L_n(X)$, and the inverse limit $\mathcal L(X) = \lim \mathcal L_n(X)$ is the full \emph{arc space} of $X$, an infinite type scheme whose $\C$-points correspond to morphisms $\Spec \C\llbracket t\rrbracket \ra X$. Kontsevich invented \emph{Motivic Integration} in order to prove that smooth birational Calabi--Yau manifolds have the same Hodge numbers; he constructed a \emph{motivic measure} on $\mathcal L(X)$, which can be thought of as the analogue of the $p$-adic measure used earlier by Batyrev to show that smooth birational Calabi--Yau manifolds have equal Betti numbers. Other remarkable notions introduced by Denef--Loeser are the \emph{motivic Milnor fibre} and the \emph{motivic vanishing cycle}; the latter is the motivic incarnation of the perverse sheaf of vanishing cycles attached to a regular (holomorphic) function $U\ra \C$. This theory has a wide variety of applications in Singularity Theory, but it has also proven successful in Algebraic Geometry, for instance in the study of degenerations of abelian varieties via motivic zeta functions \cite{Halle_2012}.

\subsection{Invincible parts}

An elegant approach to the problem of locally free replacements of principal parts has been  proposed by Patel and Swaminathan in their recent report \cite{InvParts}. Their construction is formally more adherent to the purely algebraic definition of principal parts as described in Section \ref{sec:pp}. To perform the construction they restrict to certain families of curves according to the following:

\begin{definition}
Let $\pi\colon X\ra S$ be a proper flat morphism of pure Gorenstein curves. Then $\pi$ is called an \emph{admissible family} if the locus $\Gamma\subset X$ over which $\pi$ is not smooth has codimension at least $2$.
\end{definition}

Let $\pi\colon X\ra S$ be an admissible family with $X$ and $S$ smooth, irreducible varieties, and assume $\dim S = 1$. Let $E$ be a vector bundle on the total space $X$. Patel and Swaminathan define the $k$-th order \emph{sheaf of invincible parts} associated to $(\pi,E)$ as the double dual sheaf 
\[
P^k_\pi(E)^{\vee\vee}.
\]
This intrinsic construction is related to the gluing procedure (giving rise to jets) described in Section~\ref{sec78798}, via the following observation.

\begin{prop}
Let $\pi\colon X\ra S$ be an admissible family of Gorenstein curves, with $X$ and $S$ smooth irreducible varieties and $\dim S=1$. Let $L$ be a line bundle on $X$. Then the sheaf of invincible parts $P^k_\pi(L)^{\vee\vee}$ agrees with the jet bundle $J^k_\pi(L)$ of \eqref{JetL}.
\end{prop}

\begin{proof}
The vector bundle $J^k_\pi(L)$ restricted to the smooth locus $U=X\setminus \Gamma$ of $\pi$ agrees with $P^k_\pi(L)|_U$ by Lemma \ref{lemma:smoothlocus}. But by \cite[Prop.~10]{InvParts}, $P^k_\pi(L)^{\vee\vee}$ is the unique locally free sheaf with this property.
\end{proof}

\section{Two applications}\label{sec:appl}
\subsection{Counting flexes via automatic degeneracies}

In this section we report on one of the main applications of the {\em  sheaves of invincible parts} that motivated the research by Patel and Swaminathan.
In particular, we wish to describe the application of their theory of \emph{automatic degeneracies} to the enumeration of hyperflexes in general pencils of plane curves.
A \emph{hyperflex} on a plane curve $C\subset \P^2$ is a point on the normalisation $P\in\widetilde C$ such that for some line $\ell\subset \P^2$ we have $\ord_P(\nu^\ast \ell)\geq 4$, where $\nu\colon \widetilde C\ra C$ is the normalisation map.
The general plane curve of degree $d>1$ has no hyperflexes, but one expects to find a finite number of hyperflexes in a pencil. One has the following classical result.

\begin{prop}\label{prop:hyperflexes}
In a general pencil of plane curves of degree $d$, exactly 
\[
6(d-3)(3d-2)
\]
will have hyperflexes.
\end{prop}

\begin{remark}
Note that this number vanishes for $d=3$. This should be expected, for in a general pencil of plane cubics all fibres are irreducible, but a cubic possessing a hyperflex is necessarily reducible.
\end{remark}

A proof of Proposition \ref{prop:hyperflexes} via principal parts can be found in \cite{3264}. A different approach, via relative Hilbert schemes, has been taken by Ran \cite{Ran1}. In \cite{InvParts}, the authors apply their theory of automatic degeneracies to give a new proof of Proposition \ref{prop:hyperflexes}. More precisely, after a suitable Chern class calculation, which we review below in the language of jet bundles, the authors subtract the individual contribution of each node in the pencil to get the desired answer. Strictly speaking, interpreting the node contribution in terms of automatic degeneracies is a step that relies upon a genericity assumption described neatly in \cite[Remark 20]{InvParts}, and that we take for granted here. Let us note that it is extremely useful to have an explicit function (see Subsection \ref{autdeg} below) computing the ``correction term'' one has to take into account while performing a Chern class/Porteous calculation over a family of curves containing singular members.

\subsubsection{Automatic degeneracies}\label{autdeg}
Given a (proper, non-smooth) morphism of Gorenstein curves $X\ra S$, the associated sheaves of principal parts are not locally free, but the jets constructed out of the derivation \eqref{der:dpi} are locally free. To answer questions on the inflectionary behavior of the family $X\ra S$, the classical strategy is to set up a suitable Porteous calculation and compute the degree of the appropriate Chern classes of the jet bundles. However, inflection points are by definition smooth points, and singularities in the fibres $X_s$ tend to ``attract'' inflection points as limits; so one has to excise the contribution to this Porteous calculation coming from the singular points of the fibres. This problem was tackled in \cite{InvParts}, where the authors propose a theoretical solution, working nicely at least under certain assumptions.
More precisely, the authors are able to attach to any germ $f\in \C\llbracket x,y\rrbracket$ of a plane curve singularity a function
\[
\AD(f)\colon \N\ra \N, \qquad m \mapsto \AD^m(f),
\]
whose value at $m\in \N$ they call the $m$-\emph{th order automatic degeneracy} associated to $f$. From its definition \cite[Def.~18]{InvParts}, it is clear that the function $\AD(f)$ is an \emph{analytic invariant} of the germ $f$. We refer the reader to \cite[Section 5]{InvParts} for an algorithmic approach to the computation of the values of this function.

Given a $1$-parameter admissible family $X\ra S$ of curves where the singularity $f = 0$ appears in a fibre, the number $\AD^m(f)$ is the correction term one has to take into account in the Porteous calculation aimed at enumerating $m$-th order inflection points on $X\ra S$. The authors determine this function in the nodal case by proving \cite[Theorem 21]{InvParts} the formula
\be\label{ADnode}
\AD^m(xy) = \binom{m+1}{4}.
\ee
It remains an open problem to compute the function $\AD(f)$ for other singularities, although in \emph{loc.~cit.}~a few computations for a specific $m$ are carried out, for instance
\[
\AD^4\left(y^2-x^3\right) = 10
\]
for the cusp singularity.

\subsubsection{The count of hyperflexes}\label{counthyperflex}
Let $X\subset \P^2\times \P^1\ra \P^1$ be a generic pencil of plane curves of degree $d$. It can be realised explicitly as follows. Let us choose two general plane curves $C_1$ and $C_2$ of degree $d$, the generators of the pencil. Their intersection will consist of $d^2$ reduced points. Blowing up these points gives 
\[
\pi\colon X\hookrightarrow \P^2\times \P^1\ra \P^1.
\]
Consider the line bundle $L_d = b^\ast \O_{\P^2}(d)$, where $b\colon X\ra \P^2$ is the blow up map. The number we are after is
\[
\int_X c_2(J^3_\pi(L_d))-\binom{5}{4}\cdot \delta,
\]
where $\delta = 3(d-1)^2$ is the number of nodes computed in \eqref{numberofnodes} and the binomial coefficient computes the automatic degeneracy of a node, using \eqref{ADnode} with $m = 4$.
This number is determined by the Chern classes
\[
\eta = c_1(\omega_\pi),\qquad \zeta = c_1(L_d).
\]
Using the exact sequences of Proposition \ref{prop:exactsequence} we get 
\[
c_2(J^3_\pi(L_d)) = 11 \eta^2 + 18\eta\zeta + 6 \zeta^2.
\]
It is easy to see that $\zeta^2\in A^2(X)$ has degree $1$. Exploiting that $E^2 = -d^2$, one can check that $\eta^2$ has degree $3d^2-12d+9$. Finally, $\eta\zeta$ has degree $2d-3$. The difference
\[
11(3d^2-12d+9)+18(2d-3)+6-5\cdot 3(d-1)^2 = 6(d-3)(3d-2)
\]
computes the number of hyperflexes prescribed by Proposition \ref{prop:hyperflexes}.

\subsection{The stable hyperelliptic locus in genus $3$, after Esteves}\label{estev}
In this section we will see the sheaves of principal parts and the technique of locally free replacements in action to solve a concrete problem. 
The results in this section hold over an algebraically closed field $k$ of characteristic different from $2$. Consider the moduli space $M_3$ of smooth, projective, connected curves of genus $3$. A hyperelliptic curve of genus $3$ is a $2:1$ branched covering of the projective line with $8$ ramification points.

Let $H\subset M_3$ be the divisor parametrising hyperelliptic curves, and let $\overline H$ be its closure in the Deligne--Mumford moduli space $\overline M_3$ of stable curves. The vector space $\Pic(M_3)\otimes \Q$ is generated by the Hodge class $\lambda$ (pulled back from $\overline M_3$), whereas $\Pic(\overline M_3)\otimes \Q$ is generated by $\lambda$, $\delta_0$ and $\delta_1$, with $\delta_i$ denoting the boundary classes on $\overline M_3$.
A proof of the following theorem, expressing the classes $[H]$ and $[\overline H]$ in terms of the above generators, can be found in in \cite{ModCurves1}.

\begin{teo}\label{thm:hyper3}
One has 
\be  
[H] = 9\lambda \label{eq:sgnocchera}
\ee
 and 
 \be 
 [\overline H] = 9\lambda-\delta_0-3\delta_1.\label{eq:sgnoccherac}
 \ee
\end{teo}

Formula \eqref{eq:sgnocchera} also follows from Mumford's relation \cite[p.~314]{Mumford1983}. Below is a quick description of how Esteves \cite[Thm.~1]{Esteves1} proves formula \eqref{eq:sgnoccherac}.

\subsubsection{Smooth curves}\label{smooth1732r}
Let $\pi\colon \mathcal C\ra S$ be a smooth family of genus $3$ curves. We constructed in \eqref{map:82929} a natural map of vector bundles $\nu\colon \pi^\ast\pi_\ast\Omega^1_\pi \ra P^1_\pi(\Omega^1_\pi)$ on $\mathcal C$, where the source has rank $3$ and the target has rank $2$. Assuming the general fibre is not hyperelliptic, it turns out that the top degeneracy scheme $D$ of $\nu$ (supported on points $P$ such that $\nu|_P$ is not onto) has the expected codimension, namely $2$. Then Porteous formula applies and gives
$[D] = c_2(P^1_\pi(\Omega^1_\pi)-\pi^\ast\pi_\ast\Omega^1_\pi)\cap [\mathcal C]$.
Pushing this identity down to $S$, and observing that there are $8$ Weierstrass points on a hyperelliptic curve of genus $3$, one gets, after a few calculations, the relation $8h_\pi = 72\lambda_\pi$, proving the formula for $[H]$.

\subsubsection{Stable curves}
Let now $\mathfrak X\ra S$ be a family of stable curves of genus $3$, which for simplicity we assume general from the start. This means $S$ is smooth and $1$-dimensional, the general fibre of $\pi$ is smooth and the finitely many singular fibres have only one singularity. One can see that only two types of singularities can appear in the fibres: a uninodal irreducible curve $Z\subset \mathfrak X$, or a reducible curve $X\cup_NY\subset \mathfrak X$ consisting of a genus $1$ curve $X$ meeting a genus $2$ curve transversally at the node $N$. It is also harmless to assume there is exactly one singular fibre of each type. 

After replacing the sheaf of differentials $\Omega^1_\pi$ with the (invertible) dualising sheaf $\omega_\pi$, Esteves obtains, via a certain pushout construction, a natural map of vector bundles
\[
\overline \nu\colon \pi^\ast\pi_\ast\omega_\pi\ra P^1_\pi(\omega_\pi)\ra \mathscr F
\]
where, as before, the source has rank $3$ and the target has rank $2$. Note that the middle sheaf, the sheaf $P^1_\pi(\omega_\pi)$ of principal parts, is not locally free because of the presence of singularities. However, by construction, the restriction of $\overline \nu$ to the smooth locus recovers the old map $\nu$ from the previous paragraph. Unfortunately, one cannot apply Porteous formula directly here, because this time the top degeneracy scheme of $\overline\nu$ has the wrong dimension, as it contains the elliptic component $X$.

The way out is to replace $\omega_\pi$ by its twist $L = \omega_\pi\otimes\O_{\mathfrak X}(-X)$.\footnote{A similar technique involving twisting by suitable divisors will be exploited in Section \ref{swp823}.} Repeating the pushout construction gives the diagram
\[
\begin{tikzcd}
0\arrow{r} & 
L\otimes \Omega^1_\pi \arrow{d}[swap]{\textrm{id}\otimes\phi} \arrow{r} & 
P^1_\pi(L) \arrow{d} \arrow{r} & 
L \arrow[equal]{d} \arrow{r} & 
0 \\
0 \arrow{r} & 
L\otimes \omega_\pi \arrow{r} 
& \mathscr F' \arrow{r}
& L \arrow{r}
& 0
\end{tikzcd}
\]
where $\phi$ is as in \eqref{gnocca}. The map of vector bundles
\[
\nu':\pi^\ast\pi_\ast L\ra P^1_\pi(L)\ra \mathscr F'
\]
has now top degeneracy scheme of the expected dimension. It can be characterised as follows.

\begin{prop}[{\cite[Prop.~2]{Esteves1}}]
The top degeneracy scheme $D'$ of $\nu'$ consists of:
\benum
\item the $8$ Weierstrass points of each smooth hyperelliptic fibre, each with multiplicity $1$;
\item the node of $Z$, with multiplicity $1$;
\item the node $N = X\cap Y$, with multiplicity $2$;
\item the $3$ points $A\in X\setminus \{N\}$ such that $2A = 2N$, each with multiplicity $1$;
\item the $6$ Weierstrass points of $Y$, each with multiplicity $1$.
\eenum
\end{prop}

The multiplicities tell us how much the points we do not want to count actually contribute. Esteves then proves \cite[Prop.~3]{Esteves1} the crucial relation $\pi_\ast [D'] = 72\lambda_\pi-7\delta_{0,\pi}-7\delta_{1,\pi}$. Subtracting the unwanted contributions (2) -- (5) with the indicated multiplicities on both sides, one gets the relation
\[
8\overline h_\pi = 72\lambda_\pi - 8\delta_{0,\pi} - 24\delta_{1,\pi},
\]
thus proving the formula for $[\overline H]$ in Theorem \ref{thm:hyper3}.

\section{Ramification points on Riemann surfaces} \label{sec:ram}

In order to make clear that, at least from the point of view of ramification points of linear systems, Gorenstein curves almost behave as if they were smooth, it is probably useful to quickly  introduce the notion of ramification loci of linear systems in the classical case of compact Riemann surfaces, which correspond, in the algebraic category, to smooth projective curves.

\subsection{Ramification loci of Linear Systems} 
A linear system on a smooth curve $C$ of genus $g$ is a pair $(L,V)$, where $L$ is a line bundle and $V\subset H^0(C,L)$ is a linear subspace. If $L$ has degree $d$ and $\dim V=r+1$, one refers to $(L,V)$ as a $g^r_d$ on $C$. When $V=H^0(C,L)$ the linear system is called complete. For instance the complete linear system attached to $K_C$ is the canonical linear system.
Every $g^r_d$ defines a rational map
\[
\varphi_V\colon C\dashrightarrow \P V,\qquad P\mapsto (v_0(P):v_1(P):\cdots:v_r(P)),
\]
where $(v_0,\ldots, v_r)$ is a $\C$-basis of $V$. 
The closure of the image of $\varphi_V$ is a projective curve, not necessarily smooth, of arithmetic genus $g+\delta$ where $\delta$ is a measure of the singularities of the image, that may be also rather nasty. See Proposition \ref{prophw} in the next section for the (local) meaning of the number $\delta$.
The rational map $\varphi_V$ turns into a morphism if $(L,V)$ has no \emph{base point}, that is, for all $P\in C$ there is a section $v\in V$ not vanishing at $P$. 
If moreover the map separates points, in the sense that for all pairs $P_1$, $P_2\in C$ there is a section vanishing at $P_i$ and not at $P_j$, then the map is an embedding and the image itself is smooth of the same geometric genus as $C$. For most curves a basis $(\omega_0, \ldots, \omega_{g-1})$ of $H^0(C,K_C)$ is enough to embed $C$ in $\PP^{g-1}$. The curves for which the canonical morphism is not an embedding are called {\em hyperelliptic}. They can be embedded in $\PP^{3g-4}$ by means of a basis of $K_C^{\otimes 2}$.

\medskip
We now define what it means for a section $v\in V\setminus 0$ to vanish at a point $P\in C$ to a given order. This is a crucial concept in the theory of ramification (or inflectionary behavior) of linear systems. Observe that, given a point $P\in C$, any section $v\in V$ defines an element $v_P$ in the stalk $L_P$ via the maps
\[
V\subset H^0(C,L)\ra L_P.
\]
\begin{definition}\label{def:ordvanishing}
Let $v\in V\setminus 0$ be a section, $P\in C$ a point. We define
\[
\ord_Pv \defeq \dim_\C L_P/v_P \in \mathbb N
\]
to be the \emph{order of vanishing} of $v$ at $P$.
\end{definition}

\begin{definition}\label{def:ramif}
Let $(L,V)$ be a $g^r_d$. A point $P\in C$ is said to be a \emph{ramification point} of $(L,V)$ if there exists a section $v\in V\setminus 0$ such that $\ord_Pv \geq r+1$. A ramification point of the canonical linear system $(K_C,H^0(C,K_C))$ is called a \emph{Weierstrass point}.
\end{definition}

\begin{example}\label{ex:quartic16376123}
Let $\iota\colon C\hookrightarrow \P^2$ be a smooth plane quartic. Then $C$ has genus $3$ and the complete linear system attached to $K_C = \iota^\ast\O_{\P^2}(1)$ is the linear system cut out by lines. Therefore the Weierstrass points of $C$ are precisely the \emph{flexes}. It is known classically that there are $24$ of them.
We take the opportunity here to recall that flexes of plane quartics are geometrically very relevant: their configuration in the plane determines and is determined by the smooth quartic. See the work of Pacini and Testa \cite{PaciniTesta} for this exciting story.
\end{example}

\begin{example}\label{ex:oeu222}
The $g^2_4$ on $\PP^1$ determined by 
\[
V = \CC \cdot x_0x_1^3\oplus \CC\cdot x_1^4\oplus \CC\cdot x_0^4\in G(3, H^0(\O_{\PP^1}(4)))
\]
defines the morphism $\varphi_V\colon \P^1\ra \P^2$ given by
\[
(x_0:x_1)\mapsto (x_0x_1^3:x_1^4:x_0^4).
\]
In the coordinates $x$, $y$ and $z$ on $\P^2$, the image of $\varphi_V$ is the plane quartic curve $x^4-y^3z=0$. The curve possesses a unique triple point at $P\defeq (0:0:1)$ and a hyperflex at the point $Q\defeq (0:1:0)$, as it is clear from the local equation $x^4-z=0$ (the tangent is $z=0$). An elementary Hessian calculation shows that $Q$ has multiplicity\footnote{We will soon interpret this \emph{multiplicity} as \emph{ramification weight}, see \eqref{def:weight}.} $2$ in the count of flexes of $C$. Then, by Example \ref{ex:quartic16376123}, any reasonable theory of Weierstrass points on singular curves should assign the ``weight'' $22$ to the triple point $P$, in order to reach the total number of flexes of a quartic curve. See Example \ref{Ex:quarticplane92} for the same calculation in terms of the Wronskian (cf.~also Remark \ref{ex:HessianWronskian} for the relationship between the Hessian and the Wronskian at smooth points). 

In fact, the curve $C$ can be easily smoothed in a pencil
\[
x^4-y^3z+t\cdot L(x,y)z^3=0
\]
where $L(x,y)=ax+by$ is a general linear form. An easy check, based on the computation of the Jacobian ideal, shows that the generic fibre of the pencil is a smooth quartic having a hyperflex at the point  $(0:1:0)$. Then there must be exactly $22$ smooth flexes that for $t=0$ collapse at the point $P=(0:0:1)$. {According to the theory of Widland and Lax, sketched in Section \ref{WL}, the triple point is a singular Weierstrass point of the curve, thought of as a Gorenstein curve of arithmetic genus $3$.}
\end{example}

\subsection{Gap sequences and weights}
Let $P\in C$ be an arbitrary point, $(L,V)$ a linear system, and assume $0<r<d$. For $i\geq 0$, let us denote by 
\[
V(-iP)\subset V
\]
the subspace of sections vanishing at $P$ with order at least $i$. Note that $V(-(d+1)P)=0$. If
\[
\dim V(-(i-1)P) > \dim V(-iP),
\]
then $i$ is called a \emph{gap} of $(L,V)$ at $P$. It is immediate to check that in the descending filtration
\be\label{gaps}
V\supseteq V(-P)\supseteq V(-2P)\supseteq \cdots \supseteq V(-(r+1)P)\supseteq \cdots\supseteq V(-dP)\supseteq 0
\ee
there are exactly $r+1 = \dim V$ gaps. Note that $1$ is not a gap at $P$ if and only if $P$ is a base point of $V$. 

\begin{definition}
The \emph{gap sequence} of $(L,V)$ at $P\in C$ is the sequence 
\[
\alpha_{L,V}(P):\alpha_1<\alpha_2<\cdots<\alpha_{r+1}
\]
consisting of the gaps of $(L,V)$ at $P$, ordered increasingly.
\end{definition}

For a generic point on $C$, the gap sequence is $(1,2,\dots,r+1)$, meaning that the dimension jumps in \eqref{gaps} occur as early as possible.
Equivalent to the gap sequence is the \emph{vanishing sequence}, whose $i$-th term is $\alpha_i-i$. The \emph{ramification weight} of $(L,V)$ at $P$ is the sum
\be\label{def:weight}
\wt_{L,V}(P) = \sum_i\,(\alpha_i-i).
\ee
One may rephrase the condition that $P$ is a ramification point for $(L,V)$ in the following equivalent ways:
\bitem
\item [(i)] $V(-(r+1)P) \neq 0$, that is, $(r+1)P$ is a \emph{special divisor} on $C$;
\item [(ii)] the gap sequence of $(L,V)$ at $P$ is not $(1,2,\dots,r+1)$;
\item [(iii)] the vanishing sequence of $(L,V)$ at $P$ is not $(0,0,\dots,0)$;
\item [(iv)] the ramification weight $\wt_{L,V}(P)$ is strictly positive.
\eitem

According to (i), $P\in C$ is a Weierstrass point if and only if $h^0(K_C(-gP)) > 0$.

\begin{definition}\label{ex:special}
Weierstrass points of weight one are called \emph{normal}, or \emph{simple}. On a general curve of genus at least $3$ these are the only Weierstrass points to be found.
Those of weight at least two are usually called \emph{special} (or \emph{exceptional}) Weierstrass points. 
\end{definition}

The locus in $\overline M_g$ of curves possessing special Weierstrass points has been studied by Cukierman and Diaz. We review the core computations in the subject in Section \ref{sec:specialWP}.

\subsection{Total ramification weight and Brill--Segre formulas}

The notion of ramification point of a linear system $(L,V)$ recalled in Definition~\ref{def:ramif} relies on the notion of order of vanishing of a section of $L$. This compact algebraic definition can be phrased also in the following way, which was used for the first time by Laksov \cite{Laksov1} to study ramification points of linear systems on curves in arbitrary characteristic. There exists a map 
\be\label{eq:drop}
D^r\colon C\times V \sra J^r(L),\qquad (P,v)\mapsto D^rv(P),
\ee
where $D^rv \in H^0(C,J^r(L))$ is the section defined in \eqref{troiazione}, and whose vanishing at $P$ is equivalent to the condition $\ord_Pv \geq r+1$ of Definition \ref{def:ramif}.
The map $D^r$ is a map of vector bundles of the same rank $r+1$, so it is locally represented by an $(r+1)\times (r+1)$ matrix. The condition $D^rv(P)=0$ then says that \eqref{eq:drop} drops rank at $P$. This in turn means that $P$ is a zero of the \emph{Wronskian section}
\[
\mathbb W_V \defeq \det D^r \in H^0\left(C,\bigwedge^{r+1}J^r(L)\right)
\]
attached to $(L,V)$.
The \emph{total ramification weight} of $(L,V)$, namely the total number of ramification points (counted with multiplicities), is
\[
\wt_V\defeq \deg \bw^{r+1}J^r(L) = \sum_P\wt_{L,V}(P).
\]
It can be computed by means of the short exact sequence
\[
0\sra L\otimes K_C^{\otimes r}\sra J^r(L)\sra J^{r-1}(L)\sra 0,
\]
reviewed in Proposition \ref{prop:exactsequence}. By induction, one obtains a canonical identification 
\[
\bw^{r+1}J^r(L)=L^{\otimes r+1}\otimes K_C^{r(r+1) / 2}.
\]
Using that $\deg K_C = 2g-2$, one finds the \emph{Brill--Segre formula} 
\be\label{eq:totwei}
\wt_V=(r+1)d + (g-1)r(r+1)
\ee
attached to $(L,V)$.
For instance, since $h^0(C,K_C) = g$, the number of Weierstrass points (counted with multiplicities) is easily computed as
\be\label{WPggg}
\wt_{K_C} = \deg \bigwedge^g J^{g-1}(K_C) = (g-1)g(g+1).
\ee
For $g = 3$, \eqref{WPggg} gives the $24$ flexes on a plane quartic, as in Example \ref{ex:quartic16376123}.

\section{Ramification points on Gorentein curves} \label{WL}
The study of Weierstrass points on singular curves is mainly motivated by degeneration problems. For instance it is a well known result of Diaz \cite[Appendix 2, p. 60]{Diaz2} that the node of an irreducible  uninodal curve of arithmetic genus $g$ can be seen as a limit of $g(g-1)$ Weierstrass points on nearby curves.
In this section we review the Lax and Widland construction of the Wronskian section attached to a linear system on a Gorenstein curve. 

The key idea is to define derivatives of local regular functions in the extended sense sketched at the beginning of Section~\ref{sec:sgnacchera}. One exploits the natural map $\Omega^1_C\sra \omega_C$ (see the references in Section \ref{sec78798} for its construction), where $\omega_C$ is \emph{invertible} by the Gorenstein condition. The dualising sheaf is explicitly described by means of regular differentials on $C$. Thanks to this extended definition of differential Widland and Lax are able to attach a Wronskian section to each linear system on $C$, as we shall show in Section \ref{sec:WSAWL}, after a few preliminaries aimed to reinterpret the Gorenstein condition of Definition \ref{def:gor323} in local analytic terms.   
In the last year some progress has been done also in the direction of linear systems on non-Gorenstein curves, essentially thanks to the investigations of R. Vidal-Martins. See e.g. \cite{Vidal} and references therein. As for Gorenstein curves we should mention the clever way to deform monomial curves due to Contiero and St\"ohr \cite{contierosto} to compute dimension of moduli spaces of curves possessing a Weierstrass point with prescribed numerical semigroup.

\subsection{The analytic Gorenstein condition}\label{sec:gorenstein}

Let $C$ be a Cohen--Macaulay curve. Its dualising sheaf $\omega_C$ has the properties
\be\label{eqn:gor82328}
H^0(C,\O_C) = H^1(C, \omega_C)^\vee, \qquad 
H^1(C,\O_C) = H^0(C,\omega_C)^\vee.
\ee
Recall that $g\defeq p_a(C)\defeq h^1(C, \O_C)$ is the {\em arithmetic genus} of $C$. 
For smooth curves we have $\Omega^1_C=\omega_C$. But if $C$ is singular, the sheaf $\Omega^1_C$ is no longer locally free and it does not coincide with $\omega_C$. The dualising sheaf itself may or may not be locally free: the curves for which it is are the {\em Gorenstein curves}.

\begin{example}
All local complete intersection curves are Gorenstein. This includes curves embedded in smooth surfaces as well as the \emph{stable curves} of Deligne--Mumford. Note that, by the adjunction formula, a plane curve $\iota\colon C\hookrightarrow \P^2$ of degree $d$ has canonical bundle $\omega_C=\iota^\ast \O_{\P^2}(d-3)$, clearly a line bundle. See also Example \ref{gnagna98} for a relative, more general formula.
\end{example}

The dualising sheaf $\omega_C$ of a reduced curve $C$ was first defined by Rosenlicht \cite{Rosenlicht} in terms of \emph{residues} on the normalisation of $C$.
For a Gorenstein curve, this sheaf has a very simple local description. In \cite[Section IV.10]{serre2}, to which we refer the reader for further details, it is shown that the stalk $\omega_{C,P}$ is the module of \emph{regular differentials} at $P$. 
We now recall an analytic criterion allowing one to check local freeness of $\omega_C$.

Let $\nu\colon\widetilde C\ra C$ be the normalisation of an integral curve $C$, and let $S\subset C$ be its singular locus. The canonical morphism $\O_C\ra \nu_\ast \O_{\nC}$ is injective, with quotient a finite length sheaf supported on $S$. We denote by
\be\label{deltap}
\delta_P\defeq \dim_\C \widetilde\O_{C,P}/\O_{C,P}
\ee
the fibre dimension of this finite sheaf at a point $P\in C$. Clearly $\delta_P>0$ if and only if $P\in S$. This number is an \emph{analytic} invariant of singularities \cite[p.~59]{serre2}. The sum $\sum_P\delta_P = p_a(C)-p_a(\widetilde C)$ is the number $\delta$ quickly mentioned in Section \ref{sec:ram}. Another local measure of singularities is the conductor ideal.

\begin{definition}
Let $B$ be the integral closure of an integral domain $A$. The \emph{conductor ideal} of $A\subset B$ is the largest ideal $I\subset A$ that is an ideal of $B$, that is, the set of elements $a\in A$ such that $a\cdot B\subset A$. Let $C$ be an integral curve, $P\in C$ a point. We denote by ${\mathfrak c}_P\subset \O_{C,P}$ the conductor ideal of $\O_{C,P}\subset \widetilde{\O}_{C,P}$. Define the number
\[
n_P \defeq \dim_\C \widetilde{\O}_{C,P}/\mathfrak c_P.
\] 
\end{definition}

For instance if $\widetilde{\O}_{C,P}=\O_{C,P}$ then ${\mathfrak c}_P=\widetilde{\O}_{C,P}$ and $n_P$ vanishes in this case. 
We wish to recall the following characterisation.

\begin{prop}[{\cite[Proposition IV.7]{serre2}}]
An integral projective curve $C$ is Gorenstein if and only if $n_P=2\delta_P$ for all $P\in C$.
\end{prop}

In other words, the numerical condition $n_P=2\delta_P$ guarantees that the sheaf of regular differentials is invertible at $P$.

\begin{example}
Let $P$ be the origin $(0,0)$ of the affine cuspidal plane cubic $y^2-x^3=0$. Then $\O_{C,P}={\CC[t^2,t^3]}_{(t^2,t^3)}$. The normalisation is the local ring $\widetilde{\O}_{C,P}=\CC[t]_{(t)}$. In this case the conductor is the localisation of the conductor of the subring $\C[t^2,t^3]\subset \C[t]$. Since
$$
\CC[t^2,t^3]=\CC+\CC t^2+\CC t^3+t^2\CC[t],
$$
the conductor is the ideal $(t^2,t^3)$, and its extension in $\widetilde{\O}_{C,P}$ is $(t^2)$.
Then $n_P=\dim_\CC\CC[t]/t^2=2$, and $\delta_P = \dim_\CC \CC[t]/\CC[t^2,t^3] = 1$. Thus $P$ is a Gorenstein singularity. Having this point as its only singularity, the cuspidal curve is a Gorenstein curve of arithmetic genus $1$.
\end{example}

\begin{example} Let $C$ be the complex  rational curve defined by the parametric equations $X=t^3, Y=t^4,Z=t^5$. Then $C$ is the spectrum (the set of prime ideals) of the ring $\CC[t^3,t^4,t^5]$. Clearly the origin $P=(0,0,0)$ of ${\mathbb A}^3$ is a singular point of $C$. One has that
$$
\O_{C,P}=\CC[t^3,t^4,t^5]_{(t^3,t^4,t^5)}
$$
is not a Gorenstein singularity: the conductor of $\CC[t^3,t^4,t^5]_{(t^3,t^4,t^5)}\subset \CC[t]_{(t)}$ is $t^3\CC[t]_{(t)}$. Thus  $n_P=3$, an odd number, and $C$ cannot be Gorenstein at $P$.
\end{example}

\subsection{The Wronskian section after Widland--Lax} \label{sec:WSAWL}
We now explain the construction, due to Widland and Lax, of the Wronskian attached to a linear system $(L,V)$ on a Gorenstein curve. For simplicity we shall stick to the case of \emph{integral} (reduced, as usual, and irreducible) curves to avoid coping with linear systems possessing non zero sections identically vanishing along an irreducible component. For example if $X\cup Y$ is a uninodal reducibile curve of arithmetic genus $g$ the space of global sections of the dualising sheaf has dimension $g$ but there are non-zero sections vanishing identically along $X$ (or on $Y$). However if one considers a linear system on a reducible curve that is not degenerate on any component then everything goes through just as in the irreducible case.

\medskip
If $P\in C$ is a singular point on an (integral) curve $C$, the maximal ideal $\mathfrak m_P\subset \O_{C,P}$ is not principal and so there is no local parameter whose differential would be able to freely generate $\Omega^1_{C,P}$. But we can still consider the natural map $\Omega^1_C\ra \omega_C$ (cf.~Section \ref{sec78798}) and its composition
\[
\dd\colon \O_{C}\sra \omega_{C}
\]
with the universal derivation $\O_C\sra \Omega^1_{C}$.

Let now $(L,V)$ be a $g^r_d$ on the (Gorenstein) curve $C$, and let $P$ be any point (smooth or not). Let $(v_0,v_1,\ldots, v_r)$ be a basis of $V$. Then $v_{i,P}$, the image of $v_i$ in the stalk $L_P$, is of the form $v_{i,P}=f_i\cdot \psi_P$ where $f_i\in\O_{C,P}$ and $\psi_P$ generates $L_{P}$ over $\O_{C,P}$. Letting $\sigma_P$ be a generator of $\omega_{C,P}$ over $\O_{C,P}$, one can define regular functions $f_i', f_i^{(2)}, \ldots, f_i^{(r)}\in\O_{C,P}$ through the identities
$$
\dd f_i=f_i'\cdot \sigma_P,\qquad \dd f_i^{(j-1)}=f^{(j)}\cdot \sigma_P
$$
in $\omega_{C,P}$, for each $i=0,1,\ldots,r$ (cf.~also Section \ref{sec78798}). If $P$ were a smooth point, one could take $\sigma_P=\dd z$, where $z$ is a generator of the maximal ideal $\mathfrak m_P\subset \O_{C,P}$, thus recovering the classical situation.

\begin{definition}\label{defWL}
The Widland--Lax (WL) \emph{Wronskian} around $P\in C$ is the determinant
\be\label{WLdet42}
 \WL_{V,\sigma_P}=
\begin{vmatrix}
f_0&f_1&\ldots&f_r\\ 
f'_0&f'_1&\ldots&f'_r\\ 
\vdots&\vdots&\ddots&\vdots\\
f_0^{(r)}&f_1^{(r)}&\ldots&f_r^{(r)}
\end{vmatrix}
\in \O_{C,P}.
\ee
A point $P$ is said to be a $V$-\emph{ramificaton} point (or also a $V$-Weierstrass point) if $\WL_{V,\sigma_P}(P)=0$, that is, if $\ord_P \WL_{V,\sigma_P} >0$.
\end{definition}

Our next task will be to show that the germ \eqref{WLdet42}, as well as its vanishing at $P$, does not depend on the choice of the generators $\psi_P$ and $\sigma_P$ of $L_P$ and $\omega_{C,P}$ respectively; then we will use the explicit description of  $\omega_C$ in the previous section to check that singular points are $V$-ramification points with high weight.

So if $\phi_P$ and $\tau_P$ are others generators, then $v_i=g_i\phi_P$ and $\dd g^{(i-1)}=g^{(i)}\tau_P$. Let $\psi_P=\ell_P\phi_P$ and $\sigma_P=k_P\tau_P$. Then a straightforward exercise shows that
\[
\WL_{V,\sigma_P}=\ell_P^{r+1}k_P^{r(r+1)/2}\cdot \WL_{V,\tau_P}.
\]
This proves at once that the vanishing is well defined and that  all the sections $\WL_{V,\sigma_P}$ patch together to give a global section 
\[
\WL_{V,P}\in H^0\left(C,L^{\otimes r+1}\otimes \omega_C^{\otimes r(r+1)/2}\right).
\]
If $f\in \O_{C,P}$ is any germ, according to Definition \ref{def:ordvanishing} one has 
\be\label{fsdfsq}
\ord_Pf=\dim_\CC\frac{\O_P}{f\cdot \O_P} =
\dim_\C\frac{\widetilde{\O}_P}{f\cdot \widetilde{\O}_P}=\sum_{Q\in \nu^{-1}(P)}\ord_Q f,
\ee
where in the last equality $f$ is seen as an element of $\O_{\widetilde{C},Q}$ via $\O_{C,P}\subset \widetilde \O_{C,P}\subset \O_{\widetilde C,Q}$.

\begin{definition} 
Let $P\in C$. Define the $V$-\emph{weight} of $P$ and \emph{total} $V$-\emph{ramification weight} as
\[
\wt_V(P)\defeq \ord_P\WL_{V,P},\qquad \wt_V \defeq \sum_{P\in C} \wt_V(P).
\]
\end{definition}

According to \eqref{fsdfsq}, one can compute the $V$-weight at $P$ as
\[
\wt_V(P) =
\sum_{Q\in \nu^{-1}(P)}\ord_Q\WL_{V,P}.
\]

\begin{prop}\label{prophw}
Let $(L,V)$ be a $g^r_d$ on a Gorenstein curve $C$ of arithmetic genus $g$. Then
\be
\wt_V=(r+1)d+(g-1)r(r+1).\label{eq:gorew}
\ee
Moreover, for all $P\in C$, the inequality
\be 
\wt_V(P)\geq \delta_Pr(r+1)
\ee
holds, with $\delta_P$ as defined in \eqref{deltap}. 
That is, singular points have ``high weight''.
\end{prop}

In particular if $L=\omega_C$, one has that $\wt_{\omega_C}(P)\geq \delta_Pg(g-1)$.
Proposition \ref{prophw} in \cite{LaxD} relies on an explicit description of the generator of the dualising sheaf around the singularities, that we shall review below just to provide a few examples illustrating the situation. The verification we offer here makes evident how the theory by Lax and Widland offers the right framework to study the classical Pl\"ucker formulas in terms of degenerations.

\begin{proofof}{Proposition \ref{prophw}}  Formula~(\ref{eq:gorew}) is clear. 
Let now $P$ be a singular point of $C$ and $\nu_P\colon\nC_P\sra C$ be the partial normalisation of $C$ around a singular point $P$. Then $\nC_P$ is Gorenstein of arithmetic genus $g-\delta_P$.
Consider the linear system $(\widetilde{V},\nu_P^\ast L)$, where $\widetilde{V}$ is spanned by $\nu_P^\ast v_0,\nu_P^\ast v_1,\ldots,\nu_P^\ast v_r$. It is a $g^r_d$ on $\nC_P$. Applying the formula \eqref{eq:gorew} for the total weight to $\widetilde{V}$, we find 
$$
\wt_{\widetilde{V}}=(r+1)d+({g-1-\delta_P})r(r+1)=\wt_V-\delta_Pr(r+1).
$$
The $\widetilde{V}$-Weierstrass points on $\nC_P$  are the same as the $V$-Weierstrass points on $C$. Then  the difference counts the minimum  weight of the singular point $P$ with respect to $(L,V)$.
\end{proofof}

In general $\wt_V(P)=\delta_Pr(r+1)+E(P)$. The correction $E(P)$ is called the \emph{extraweight}. It is zero if no point of $\nu_P^{-1}(P)$ is a ramification point of the linear system $(\widetilde{V}, \nu_P^*L)$. 

\begin{example}\label{weightcusp24}
If $P\in C$ is a cusp, one has $\delta_P=1$, hence its weight is at least $r(r+1)$. However the vanishing sequence of $\widetilde{V}$ at the preimage of $P$ in the normalisation is $0,2,\ldots,r+1$. Then
$$
\wt_V(P)=r(r+1)+r=r(r+2).
$$
If $L=\omega_C$ then $\wt_{\omega_C}(P)=g^2-1$.
\end{example}

\medskip
Before offering a few examples of how the $\WL$ Wronskian works concretely in computations, we recall the following fact.

\begin{prop}[{\cite[p.~362]{GarciaLax}}]
Let  $\tau$ be a local section of  \, $\Omega^1_{\widetilde{C}}$ and let $\tau_Q$ its image in the stalk  $\Omega^1_{\widetilde{C},Q}$. Assume that  $\Omega^1_{\widetilde{C}, Q}=\O_{\widetilde{C},Q}\cdot \tau_Q$ for all $Q\in \nu^{-1}(P)$ and that $h$ generates the conductor in each local ring $\O_{\widetilde{C}, Q}$. Then $\tau/h$ generates $\omega_{C,P}$ over $\O_{C,P}$.
\end{prop}

\begin{example} \label{Ex:quarticplane92}
Let us revisit from Example \ref{ex:oeu222} the rational irreducible quartic plane curve given by $x^4-y^3z=0$ in homogeneous coordinates $x$, $y$, $z$ on $\P^2$. It is Gorenstein of arithmetic genus $3$ with $\omega_C=\O_ {\PP^2}(1)|_{C}$. It has a triple point at $P \defeq (0:0:1)$ and a hyperflex at $Q \defeq (0:1:0)$, i.e.~a Weierstrass point of weight $2$. To see that the Weierstrass weight at $Q$ is $2$ one may argue by writing down the Wronskian of a basis of holomorphic differentials adapted at $Q$ (i.e.~$\omega_0=\dd t$, $\omega_1=t\dd t$ and $\omega_2=t^4\dd t$). The vanishing sequence is $0,1,4$ (equivalently, the gap sequence is $1,2,5$) so the weight is $2$.

In the chart $z\neq 0$, $V=H^0(C,\omega_C)$ is spanned by $(t^3,t^4,1)$, which are nothing but the parametric equations mapping $\PP^1$ onto the quartic. One has
\[
\O_{C,P}=\CC+\CC\cdot t^3+\CC\cdot t^4+ t^6\cdot \CC[t]_{(t)},\qquad n_P = 6,\qquad \delta_P=3.
\]
According to Proposition \ref{prophw}, $P$ is a Weierstrass point with weight at least $\delta_P\cdot 3(3-1)=18$. The exact  weight can be directly computed through the Wronskian as follows. 
The preimage of $P$ through the normalisation map is just one point $\widetilde P$. Then $\dd t$ generates $\Omega^1_{\widetilde C,\widetilde P}$ and therefore 
$\sigma=\dd t/t^6$ is a regular differential at $P$. A basis of the space of regular differentials at $P$ is then given by
\[
(\sigma, t^3\sigma, t^4\sigma),
\]
so the Wronskian is
\[
\begin{vmatrix}
1&t^3&t^4\cr
0&3t^8&4t^9\cr
0&24t^{13}&36t^{14}
\end{vmatrix}\in t^{22}\cdot \CC[t].
\]
It follows that $P$ is a Weierstrass point of weight 22, as anticipated in Example \ref{ex:oeu222}. Together with the hyperflex at $Q$, one fills the total weight, $24$, of a Gorenstein curve of genus $3$. The example shows that the point $P$ has extraweight $E(P) = 4$. This can also be computed by looking at the vanishing sequence of the linear system $\widetilde{V}$, generated by $(1,t^3,t^4)$. Clearly the vanishing sequence is $0,3,4$, whose weight is $4$, as predicted by the calculation above.

The output of this example is of course in agreement with the classical fact that the Hessian of the given plane curve cuts the singular points and the flexes. In this case the Hessian cuts indeed the singular point with multiplicity $22$ and the hyperflex $Q$ with multiplicity $2$.
\end{example}

\begin{remark}\label{ex:HessianWronskian}
A local calculation shows that the Hessian of a plane curve cutting the inflection points with respect the linear system of lines follows by the vanishing of the Wronskians at those points (at least when they are smooth).
\end{remark}

\begin{example}
The previous example was rather easy because we have dealt with a unibranch singularity (that is, $\nu^{-1}(P)$ consisted of just one point). To illustrate the behavior of the $\WL$ Wronskian with multibranch singularities, let $C$ be the plane cubic $x^3+x^2z-y^2z = 0$. It has a unique singular point, the node $P\defeq (0:0:1)$. The curve $C$ is Gorenstein of arithmetic genus $1$. Let us compute its $V$-weight, where $V$ denotes the complete linear system $H^0(C,\O_{\P^2}(1)|_{C})$. Clearly the coordinate functions $x$, $y$ and $z$ form a basis of $V$. They can be expressed by means of a local parameter $t$ on the normalisation $\nu\colon\P^1\ra C$. In the open set $z\neq 1$, indeed, $C$ has parametric equations $x=t^2-1$ and $y=t(t^2-1)$. The preimage of the point $P$ via $\nu$ are $Q_1\defeq (t-1)$ and $Q_2\defeq (t+1)$ thought of as points of $\Spec \C[t]$. One has that $\O_{C,P}=\CC+(t^2-1)\cdot \widetilde{\O}_{C,P}$, thus the conductor is $(t^2-1)$. Since $\dd t$ generates both $\Omega^1_{Q_1}$ and $\Omega^1_{Q_2}$, $\sigma_P\defeq \dd t/(t^2-1)$ generates the dualising sheaf $\omega_{C,P}$.
Let
$$
\sigma_{Q_1}\defeq\frac{\dd t}{t-1}\qquad \mathrm{and}\qquad \sigma_{Q_2}\defeq\frac{\dd t}{t+1}.
$$
Then one has
$$
\wt_V(P)=\ord_P\WL_{V,\sigma_P}=\ord_{Q_1}\WL_{V,\sigma_{Q_1}}+\ord_{Q_2}\WL_{V,\sigma_{Q_2}}.
$$
We shall show that 
$
\ord_{Q_1}\WL_{V,\sigma_{Q_1}}=3,
$
By symmetry, the same will hold for 
$
\ord_{Q_2}\WL_{V,\sigma_{Q_2}},
$
showing that the weight of $P$ as a singular ramification point is $6$ as expected.
For simplicity, let us put $z=t-1$. In this new coordinate the basis of $\nu^*L$ near $Q_1$ is given by $v_0\defeq z(z+2)$, $v_1\defeq z(z^2+3z+2)$ and $v_3=1$. The conductor is generated by $z$ near $Q_1$. Then the $\WL$-Wronskian near $Q_1$ is:
\[
\WL_{V,\sigma_{Q_1}}=
\begin{vmatrix}z^2+2z&z^3+3z^2+2z&1\cr
2z^2+2z&3z^2+6z^2+2z&0\cr 
4z^2+2z&6z^2+12z^2+2z&0
\end{vmatrix}=z^3(3z+4)\in z^3\cdot \CC[z]
\]
as desired. The computations around $Q_2$ are similar and then $P$ is a singular ramification point of weight $6$.
\end{example}

\section{The class of special Weierstrass points}\label{sec:specialWP}

\subsection{Introducing the main characters}
Let $M_g$ be the moduli space of smooth projective curves of genus $g\geq 2$. It is a normal quasi-projective variety of dimension $3g-3$. Let 
\[
M_g\subset \overline M_g
\] 
be its Deligne--Mumford compactification via \emph{stable} curves. It is a projective variety with orbifold singularities. Thus, its Picard group with rational coefficients is as well-behaved as the Picard group of a smooth variety. 
The boundary $\overline M_g\setminus M_g$ is a union of divisors $\Delta_i\subset \overline M_g$, each obtained as the image of the clutching morphism 
\[
\overline M_{i,1}\times \overline M_{g-i,1}\ra \overline M_g,
\]
defined by glueing two stable $1$-pointed curves $(X,x)$ and $(Y,y)$ identifying the markings $x$ and $y$. By a general point of $\Delta_i$ we shall mean a curve that lies in the image of the open part $M_{i,1}\times M_{g-i,1}$.
Note that $i$ ranges from $0$ to $[g/2]$, with $i = 0$ corresponding to irreducible uninodal curves. We use the standard notation $\delta_i$ for the class of $\Delta_i$ in $\Pic(\overline{M}_g)\otimes\Q$, and we always assume $i\leq g-i$.

\begin{figure}[h]
\begin{tikzpicture}[scale=0.85] 
\draw (-2,0) to[bend right] (0.3,1);
\draw (-0.3,1) to[bend right] (2,0);
\node at (0,0.62) {$\cdot$};
\node at (-2.3,0) {\small{$X$}};
\node at (2.3,0) {\small{$Y$}};
\node at (0.01,0.93) {\tiny{$A$}};
\node at (-1.1,0.3) {\small{$i$}};
\node at (1.35,0.3) {\small{$g-i$}};
\end{tikzpicture}\caption{A general element of the boundary divisor $\Delta_i\subset \overline M_g$.}
\end{figure}
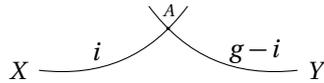

This section aims to sketch the calculation of the class in $\Pic(\overline{M}_g)\otimes\Q$ of the closure in $\overline{M}_g$ of the locus of points in $M_g$ corresponding to curves possessing a special Weierstrass point. Recall from Definition \ref{ex:special} that a Weierstrass point (WP, for short) is \emph{special} if its weight as a zero of the Wronskian is strictly bigger than $1$. Let us define
\be\label{eq:wtk}
\wt(k)\defeq \Set{[C]\in M_g | \textrm{$C$ has a WP with weight at least }k}.
\ee

\medskip
Let $M_{g,1}$ be the space of $1$-pointed smooth curves, and $\overline M_{g,1}$ be the moduli space of stable $1$-pointed curves.
Borrowing standard notation from the literature, define the ``vertical'' loci
\begin{align*}
\textrm V\mathbb D_{g-1} &\defeq 
\Set{[C,P]\in M_{g,1}|P\textrm{ is a WP whose first non-gap is }g-1}
\\
\textrm V\mathbb D_{g+1} &\defeq \Set{[C,P]\in M_{g,1}| \textrm{there is }\sigma\in H^0(C,K_C) \textrm{ such that }D^g\sigma(P)=0}.
\end{align*}
Taking their images along the forgetful morphism $M_{g,1}\sra M_g$ we get the subvarieties $\mathbb D_{g-1}$ and $\mathbb D_{g+1}$ of $M_g$, respectively. Diaz \cite[Section 7]{Diaz2} and Cukierman \cite[Section~5]{Cuk1} were able to determine the classes
\[
\left[\overline{\mathbb D}_{g\pm 1}\right]\in \Pic(\overline M_g)\otimes\Q. 
\]
The main observation of \cite{Gatto1} is that while computing the classes of $\overline{\mathbb D}_{g\pm 1}$  is quite hard, the computation of their sum is quite straightforward. Let 
\[
\overline{\textrm V\wt(2)}\subset \overline{M}_{g,1}
\]
be the closure of the locus of points $[C,P]\in M_{g,1}$ such that $P$ is a special Weierstrass point on $C$, namely a zero of the Wronskian of order bigger than $1$. The goal is to globalise the notion of Wronskian to families possessing singular fibres. This will be achieved through jet extensions of the relative dualising sheaf defined on a family of stable curves. Using (a) the invertibility of the relative dualising sheaf and (b) the locally free replacement of the principal part sheaves for such families, everything goes through via a standard Chern class calculation, as we show below.
We warn the reader that our computation is not performed on the entire moduli space but just on $1$-parameter families of stable curves with smooth generic fibre, in order to avoid delicate foundational issues regarding the geometry of the moduli space of curves.

\subsection{Special Weierstrass points} \label{swp823}
Let $\pi'\colon\Ccal'\sra T$ be a (proper, flat) family of stable curves over a smooth projective curve $T$, such that $\Ccal'$ is a smooth surface,  with smooth generic fibre $\mathcal C'_\eta$. In particular, by the compactness of $T$, the fibre $\mathcal C'_t$ is smooth for all but finitely many $t\in T$. 
If the family is general, the singular fibres are general curves of type $\Delta_i$. The general fibre of type $\Delta_0$ is an irreducible uninodal curve of arithmetic genus $g$. Let $\pi\colon\Ccal\sra T$ be the family one gets by blowing up all the nodes of the irreducible singular curves. The irreducible nodal fibres get replaced by curves of the form $C\cup L$, where $C$ is a smooth irreducible curve of genus $g-1$ and $L$ is a smooth rational curve, intersecting $C$ transversally at two points (the preimages of the node through the blow up map). The rational component $L$ is the exceptional divisor which contracts onto the node by blow down. From now on we shall work with the new family 
\[
\pi\colon\Ccal\sra T,
\]
where all the singular fibres are reducible.

As for all families of stable curves, the dualising sheaf $\omega_\pi$ is invertible, and its pushforward
\[
\mathbb E_\pi\defeq \pi_\ast\omega_\pi
\]
is a rank $g$ vector bundle on $T$, called the \emph{Hodge bundle} (of the family). Its fibre over $t\in T$ computes
\[
H^0(\mathcal C_t,\omega_\pi|_{\mathcal C_t}) = H^0(\mathcal C_t,\omega_{\mathcal C_t}).
\] 
If $\Ccal_{t_0} = X\cup_AY$ is a uninodal reducible curve of type $\Delta_i$, one has a splitting
\be\label{formula:8237283}
H^0(\Ccal_{0}, \omega_{\Ccal_{0}})=H^0(X,K_X(A))\oplus H^0(Y, K_Y(A)).
\ee

A Weierstrass point on the generic fibre is a ramification point of the complete linear series attached to $K_{\Ccal_\eta}\defeq\omega_\pi|_{\Ccal_\eta}$. So it must belong to the degeneracy locus of the map of rank $g$ vector bundles
\[
\begin{tikzcd}
\pi^\ast\EE_\pi \arrow{rr}{\mathsf D^{g-1}}\arrow[dash]{dr} & &
J_\pi^{g-1}(\omega_\pi) \arrow[dash]{dl} \\
& \Ccal \arrow{d}{\pi} & \\
& T &
\end{tikzcd}
\]
The zero locus of the determinant map $\bw^g \mathsf D^{g-1}$
may be identified with a section $\mathbb W_\pi$ of the line bundle
$$
\mathscr L \defeq \bw^g J_\pi^{g-1}(\rds)\otimes \pi^\ast \bw^g\EE_\pi^\vee.
$$ 
The vanishing locus of this section cuts the Weierstrass points on the generic fibre. Moreover  $\mathbb W_\pi$ identically vanishes on the reducible fibres $\mathcal C_t$ of type $\Delta_i$ for $1\leq i\leq [g/2]$. Indeed, the identification \eqref{formula:8237283} shows that there exist nonzero regular differentials on $\mathcal C_t$ vanishing identically on either component.
Moreover $\mathbb W_\pi$ identically vanishes on the rational components $L$ gotten by blowing up the nodes of the original irreducible nodal fibres. 

A local computation due to Cukierman \cite[Proposition 2.0.8]{Cuk1} (but see also \cite{CuGaEs} for an alternative way of computing), determines the order of vanishing of $\wrp$ along  each component of the reducible fibres of $\pi$. Let $F\subset \mathcal C$ be the Cartier divisor corresponding to the zero locus of $\mathbb W_\pi$ along the singular fibres. Then, letting $\mathsf Z_\eta$ be the cycle representing $\overline{Z(\wrp|_{\Ccal_\eta})}\subset \mathcal C$, one has
\[
\left[Z(\wrp)\right]=\mathsf Z_\eta + F.
\]
One can view $\overline{Z(\wrp|_{\Ccal_\eta})}$ as the zero locus of the Wronskian section ``divided out'' by the local equations of the components of the singular fibres. More precisely, $\wrp$ induces a section $\widetilde{\mathbb W}_\pi$ of the line bundle $\mathscr L(-F)$, which coincides with $\wrp$ away from $F$.
Therefore we have
\be\label{robamenoF}
\mathsf Z_\eta =
c_1\left(\bw^g J^{g-1}_\pi(\omega_\pi)\right)-\pi^*c_1(\EE_\pi)-F
= \frac{1}{2}g(g+1)c_1(\rds)-\pi^*\lambda_\pi-F
\ee
where $\lambda_\pi\defeq c_1(\EE_\pi)$ denotes, as is customary, the first Chern class of the Hodge bundle of the family. From now on, we use the (standard) notation $K_\pi\defeq c_1(\omega_\pi)$.

\begin{remark}
Intersecting the class \eqref{robamenoF} with a fibre $\Ccal_t$, one gets
$$
\mathsf Z_\eta \cdot \Ccal_t=\frac{1}{2}g(g+1)K_\pi\cdot \mathcal C_t-\pi^\ast\lambda_\pi\cdot \Ccal_t -F\cdot \Ccal_t.
$$
But the second and third  products vanish because $\Ccal_t$ is linearly equivalent to the generic fibre (and the intersection of two fibres is zero), whereas the first term corresponds to a divisor of degree $(g-1)g(g+1)$ on $\mathcal C_t$. In the case where $t$ corresponds to a singular fibre, the degree of this divisor would be the total weight of the {\em limits} of Weierstrass points on that fibre.
\end{remark}

The issue is now to detect and compute the class of the locus of special Weierstrass points in the fibres of $\pi$. Since the family $\pi$ may have singular fibres, the traditional version of principal parts would not help unless one decided to focus on open sets where they are locally free. This is for example the approach followed in \cite{Cuk1}. However, using the locally free replacement provided by jet bundles, we can now consider the ``derivative'' $D\widetilde{\mathbb W}_\pi$ of the section $\widetilde{\mathbb W}_\pi\in H^0(\mathcal C,\mathscr L(-F))$, where $\mathscr L(-F))$ denotes the twist $\mathscr L\otimes_{\O_{\mathcal C}}\O_{\mathcal C}(-F)$. The derivative $D\widetilde{\mathbb W}_\pi$ is a global holomorphic section of the rank two bundle
\[
J^1_\pi(\mathscr L(-F)).
\]
By abuse of notation let us write simply $\overline{\textrm V\wt(2)}$ for the locus  $\overline{\textrm V\wt(2)}_\pi\subset \mathcal C$ defined by the zero locus of $D\widetilde{\mathbb W}_\pi$.

\begin{definition}
Let $C_0$ be any stable curve of arithmetic genus $g\geq 2$. A point $P_0\in C_0$ is said to be a limit of a (special) Weierstrass point if there exists a family $\TS\sra \Spec\C\llbracket t\rrbracket$ such that $\TS_\eta$ is smooth, $\TS_0$ is semistably equivalent to $C_0$ and there is a (special) Weierstrass point $P_\eta$ such that $P_0\in \overline{P_\eta}$.
\end{definition}

It turns out that  $\overline{\textrm V\wt(2)}$ is the locus of special Weierstrass points on smooth fibres of $\pi$. In fact if the family $\Ccal\sra T$ is general, then only singular fibres of the codimension $1$ boundary strata of $\overline{M}_g$ occur. If $X\cup_A Y$ is a general member of $\Delta_i$,  one may assume that $A$ is not a Weierstrass point neither for $X$ nor for $Y$. Then if $P_0\in X\subset X\cup_AY$ is a limit of a special Weierstrass point it must be a special ramification point of $K_X((g-i+1)P)$ by \cite[Theorem 5.1]{CuGaEs}. But by \cite{CuGaEs1}, for a general curve $X$ and for each $j\geq 0$,  there are only finitely many pairs $(P,Q)\in X\times X$ such that $Q$ is a special ramification point of the linear system $K_X((j+1)P)$. See also Example \ref{excusp} below.

It follows that the locus $\overline{\textrm V\wt(2)}$ is zero dimensional. Indeed, the  special Weierstrass points have the expected codimension $2$ in general family of smooth curves. Its class is given by the top (that is, second) Chern class of $J^1_\pi(\mathscr L(-F))$. Explicitly, we have
\be
\left[\overline{\textrm  V\wt(2)}\right]=c_2\left(J^1_\pi\left(\omega_\pi^{\otimes g(g+1)/2}\otimes\pi^*\bw^g\EE_\pi^\vee(-F)\right)\right).
\ee
By the Whitney sum formula applied to the short exact sequence
\[
0\ra \omega_\pi\otimes \mathscr L(-F)\ra J^1_\pi(\mathscr L(-F))\ra \mathscr L(-F)\ra 0,
\]
and recalling that \eqref{robamenoF} is computing precisely $c_1(\mathscr L(-F))$, one finds
\[
\left[\overline{\textrm V\wt(2)}\right] = 
\left(\frac{1}{2}g(g+1)K_\pi-\pi^\ast\lambda_\pi-F\right)\left(\frac{1}{2}g(g+1)K_\pi+K_\pi-\pi^\ast\lambda_\pi-F\right).
\]
Thus in $A^2(\mathcal C)$ we find
\[
\left[\overline{\textrm V\wt(2)}\right] = \frac{1}{4}g(g+1)(g^2+g+2)K_\pi^2
- (g^2+g+1)(K_\pi(F+\pi^*\lambda_\pi))
+ F^2,
\]
where we have used $(\pi^*\lambda_\pi)^2=0=F\cdot \pi^*\lambda_\pi$.
We want to compute the pushforward
\begin{multline}\label{eq:pistarvt}
\pi_\ast\left[\overline{\textrm V\wt(2)}\right] = \frac{1}{4}g(g+1)(g^2+g+2)\pi_\ast K_\pi^2 \\
-(g^2+g+1)\bigl(\pi_*(K_\pi\cdot  F)
+\pi_*(K_\pi\cdot \pi^*\lambda_\pi)\bigr)
+\pi_*F^2.
\end{multline}

The reason why we are interested in the class \eqref{eq:pistarvt} is that if $g\geq 4$ the degree of $\pi$ restricted to $\textrm V\wt(2)$ is $1$. Therefore, if we let 
\[
\overline{\wt(2)}\subset T
\]
be the locus of points parametrising fibres possessing special Weierstrass points, then its class is given by \eqref{eq:pistarvt}. 
The reason why for $g\geq 4$ the degree of $\pi$ is $1$, is because of the following important result, obtained by combining results by Coppens \cite{Coppens} and Diaz \cite{Diaz4}.

\begin{teo} \label{teoCopDiaz}
If a  general curve of genus $g\geq 4$ has a special Weierstrass point, then all the other points are normal.
\end{teo}

To complete the computation, let $F_i\subset \mathcal C$ be the (vertical) divisor corresponding to the zero locus of the Wronskian along the singular fibres of type type $\Delta_i$, for $1\leq i\leq [g/2]$. 
Thus $F=\sum_{i=1}^{[g/2]} F_i$ and clearly we have $F_{i_1}\cdot F_{i_2}=0$ for $i_1\neq i_2$. Moreover, we have decompositions
\[
F_i\defeq \sum_j F_{ij}, \qquad F_{ij}=m_iX_{j}+m_{g-i}Y_j,
\]
with each $F_{ij}$ supported on a fibre $X_j\cup_{A_j}Y_j$ of type $\Delta_i$. Recall that the notation means that $X_j$ and $Y_j$ have genus $i$ and $g-i$ respectively, and they meet transversally at the (unique) node $A_j$. The multiplicities $m_i$ (resp.~$m_{g-i}$) with which $\wrp$ vanishes along $X_j$ (resp.~$Y_j$) only depend on $i$. Using that $-Y_j^2=-X_j^2=X_j\cdot Y_j=[A_j]\in A^2(\mathcal C)$, it is easy to check that
\[
F_{ij}^2=\left(2m_im_{g-i}-m_i^2-m_{g-i}^2\right)[A_j].
\]

To compute \eqref{eq:pistarvt}, we will apply the projection formula $\pi_\ast(\pi^\ast\alpha\cdot \beta)=\alpha\cdot \pi_\ast\beta$.
The pushforward $\pi_*K^2_\pi$ is by definition the tautological class $\kappa_1\in A^1(T)$.  Define 
\[
\delta_{i,\pi} \defeq \sum_j \pi_\ast[A_j]\in A^1(T).
\]
This is the class of the points corresponding to singular fibres of type $\Delta_i$. We have the following identities in $A^1(T)$:
\begin{align*}
\pi_*(K_\pi\cdot \pi^*\lambda_\pi) &= 
\pi_*K_\pi\cdot\lambda_\pi=(2g-2)\lambda_\pi \\
\pi_*(K_\pi\cdot F_{ij}) &= m_i\pi_*(K_\pi\cdot  X_j)+m_{g-i}\pi_*(K_\pi\cdot  Y_j) \\
&= \left(m_i(2i-1)+m_{g-i}(2(g-i)-1)\right)\cdot \pi_\ast [A_j] \\
&= \left(2(im_i+(g-i)m_{g-i})-m_i-m_{g-i}\right)\cdot \pi_\ast [A_j].
\end{align*}

Substituting the above identities in \eqref{eq:pistarvt} we obtain
\begin{multline}\label{eq2:pistvt}
\pi_*\left[\overline{\textrm V\wt(2)}\right] = \frac{1}{4}g(g+1)(g^2+g+2)\kappa_1
-2(g^2+g+1)(g-1)\lambda_\pi
-c_0\delta_{0,\pi}-\sum_{i=1}^{[g/2]}c_i\delta_{i,\pi}
\end{multline}
where $\delta_0$ is the class of the locus in $T$ of type $\Delta_0$ (irreducible uninodal), $c_0$ is a coefficient to be determined and 
\begin{equation}\label{eq:Cmxy}
c_i=(g^2+g+1)\left(2(im_i+(g-i)m_{g-i})-m_i-m_{g-i}\right)
+2m_im_{g-i}-m_i^2-m_{g-i}^2.
\end{equation}
Now one uses one of the most fundamental relations between tautological classes. The class $\kappa_{1,\pi}$ and $\lambda_\pi$ are not independent: they satisfy the relation
\[
\kappa_{1,\pi}=12\lambda_{\pi}-\sum_{i}\delta_{i,\pi}.
\]
This is a conseguence of the Grothendieck--Riemann--Roch formula, as explained for instance in \cite{MumfordStability}. Thus formula~(\ref{eq2:pistvt}) can be simplified into
\begin{multline}\label{676762}
\pi_\ast\left[\overline{\textrm V\wt(2)}\right] = \left(3g(g+1)(g^2+g+2)-2(g^2+g+1)(g-1)\right)\lambda_\pi\\
-\sum_{i=0}^{[g/2]}\left(c_i+\frac{1}{4}g(g+1)(g^2+g+2)\right)\delta_{i,\pi},
\end{multline}
which, after renaming coefficients, becomes
\begin{equation}\label{eq2;prefinal}
\pi_\ast\left[\overline{\textrm V\wt(2)}\right] = \left(2+6g+9g^2+4g^3+3g^4\right)\lambda_\pi 
- a_0\delta_0- \sum_{i=1}^{[g/2]}b_i\delta_{i,\pi}.
\end{equation}
Clearly the expression \eqref{eq2;prefinal} is not complete: one still needs to determine the coefficients $a_0$ and $b_i$. Computing $b_i$ amounts to finding the explicit expressions for $m_i$, for all $1\leq i\leq [g/2]$. This has been done by Cukierman in his doctoral thesis (but see \cite[Proposition 6.3]{CuGaEs} for an alternative slightly more conceptual, although probably longer, proof).

\begin{teo}[{\cite[Prop.~2.0.8]{Cuk1}}]\label{cukieteo}
The multiplicities $m_i$  with which the Wronskian $\wrp$ vanishes along $X_j$ (of genus $i$), are given by:
\be\label{eq3:mxy}
m_i = \binom{g-i+1}{2}.
\ee
\end{teo}
The way Cukierman proves Theorem~\ref{cukieteo} is the following. He considers a family $f\colon\TS\sra S$ of curves of genus $g$ parametrised by $S=\Spec\C\llbracket t\rrbracket$, with smooth generic fibre and special fibre semistably equivalent to a uninodal reducible curve $X\cup_AY$ with components of genus $i$ and $g-i$ respectively. After checking that $f_\ast\omega_f\otimes k(0)$ is isomorphic to $H^0(K_X(A))\oplus H^0(K_Y(A))$, he constructs suitable global bases of $f_\ast\omega_f$ such that the first elements are non degenerate on one component and vanish on the other. He then computes the relative Wronskian using such bases and finds the multiplicity displayed in \eqref{eq3:mxy}. All the technical details are in \cite{Cuk1}. 

\medskip
Granting Theorem \ref{cukieteo}, we can now compute the right hand side of \eqref{eq2;prefinal}. We need to substitute the expressions \eqref{eq3:mxy} into the constant $c_i$ defined in \eqref{eq:Cmxy}. 
This finally gives (see also \cite{Gatto1} for more computational details)
\be\label{fwfw343}
b_i=(g^3+3g^2+2g+2)i(g-i).
\ee
We still have to determine $a_0$. To this end, we use the following argument, due to Harris and Mumford \cite{HaMu}. Consider the simple elliptic pencil $x_0E_1+x_1E_2$, where $E_1$ and $E_2$ are two plane cubics intersecting transversally at $9$ points. Let $\Scal$ be the blow-up of $\PP^2$ at the intersection points. This gives an elliptic fibration
\be\label{ellfb}
\epsilon\colon \Scal\sra \PP^1
\ee
with nine sections (the exceptional divisors of the blown up points. Let $\Sigma_1$ be any one of them. Then consider a general curve $C$ of genus $g-1$, and choose a constant section $P\colon C\sra C\times C$. Construct the family $\phi\colon\Fcal_1\sra \PP^1$, by gluing $C\times C$ and $\Scal$, by identifying $\Sigma_1$ with $P$. The fibre over a point $t\in\P^1$ is the union $C\cup E_t$, with $C$ meeting $E_t = \epsilon^{-1}(t)$ transversally at a single point. In other words, what varies in the family is just the $j$-invariant of the elliptic curve.

\begin{teo}[{\cite[Lemma 7.2]{Diaz2}}]
The fibres of $\phi\colon\Fcal_1\sra \PP^1$ contain no limits of special Weierstrass points, that is, $\phi_\ast[\overline{\mathrm V\wt(2)}]=0$.
\end{teo}

Harris and Mumford computed the degrees of $\lambda$, $\delta_0$ and $\delta_1$ to be, respectively: $1$, $12$ and $-1$. Taking degrees on both sides of \eqref{eq2;prefinal}, with $\phi$ taking the role of $\pi$, we get the (numerical) relation
\[
0=\int_{\P^1}\phi_*\left[\overline{\textrm V\wt(2)}\right] =(2+6g+9g^2+4g^3+3g^4)\cdot 1 -a_0\cdot 12+b_1\cdot 1.
\]
Given the expression of $b_1$ computed in \eqref{fwfw343}, one obtains
\[
a_0=\frac{1}{6}g(g+1)(2g^2+g+3).
\]
We have therefore reconstructed the proof of the following result.

\begin{teo}[{\cite[Theorem 5.1]{Gatto1}}]
Let $\pi\colon \Ccal\sra T$ be a family of stable curves of genus $g\geq 4$ with smooth generic fibre. Then the class in $A^1(T)$ of the locus of points whose fibres possess a special Weierstrass point is
\begin{multline}\label{eq:wt2}
\pi_*\left[\overline{\mathrm V\wt(2)}\right]=\left(2+6g+9g^2+4g^3+3g^4\right)\lambda_\pi\\
-\frac{1}{6}g(g+1)(2g^2+g+3)\delta_0
-\sum_{i=1}^{[g/2]} (g^3+3g^2+2g+2)i(g-i)\delta_i.
\end{multline}
\end{teo}

\begin{remark}
Let now $[\overline{\wt(2)}]$ be the class in $A^1(T)$ of the locus of points of $T$ corresponding to fibres carrying special Weierstrass points. By Theorem~\ref{teoCopDiaz}, for $g\geq 4$ one has
\[
\left[\overline{\wt(2)}\right]=\deg(\pi)\left[\pi(\overline{\textrm V\wt(2)})\right]=\left[\pi(\overline{\textrm V\wt(2)})\right]=\pi_*\left[\overline{\textrm V\wt(2))}\right],
\]
because $\deg(\pi)=1$.
We may conclude that for $g\geq 4$, the right hand side of \eqref{eq:wt2} is the expression of the class $[\overline{\wt(2)}]$.
\end{remark}

\subsection{Low genus}
We observe that formula \eqref{eq:wt2} holds for genus $1$, $2$ and $3$ as well, and actually recovers classical relations among tautological classes.

\subsubsection{Genus $1$}
Recall the elliptic fibration $\epsilon$ from \eqref{ellfb}.
No member of the pencil (either a smooth or rational plane cubic) possesses Weierstrass points. In particular there are no special Weierstrass points. Then $[\overline{\wt(2)}]=0$. Setting $g=1$ in \eqref{eq:wt2} one obtains the relation 
\be\label{hjh2jk13}
12\lambda-\delta_0=0,
\ee
expressing the classical fact that $\epsilon\colon\mathcal S\ra \P^1$ has $12$ irreducible nodal fibres. Indeed, the degree of $\lambda$ on this pencil is $1$, as the relative dualising sheaf restricted to the section $\Sigma_1\subset \mathcal S$ is $\O_{\mathcal S}(-\Sigma_1)|_{\Sigma_1}$, which has degree $-\Sigma_1^2 = 1$.

\subsubsection{Genus $2$}
A curve of genus $2$ is hyperelliptic: it is a ramified double cover of the projective line. The Riemann--Hurwitz formula gives $6$ ramification points which are the Weierstrass points. All these ramification points are simple. This means that if $\Ccal\sra T$ is a family of curves of genus $2$, then
\be
0=\left[\overline{\wt(2)}\right]=130\lambda-13\delta_0-26\delta_1.\label{eq:Charr}
\ee
This recovers the well known relation $10\lambda-\delta_0-2\delta_1=0$, discussed in \cite{Mumford1983}, showing that the classes $\lambda,\delta_0,\delta_1$ are not independent in $\Pic(\overline{M}_2)\otimes\Q$. See \cite{Cornalba1988} for the generalisation and \cite{GattoEsteves} for the interpretation of the Cornalba and Harris formula generalising \eqref{eq:Charr} in the rational Picard group of moduli spaces of stable hyperelliptic curves.

\subsubsection{Genus $3$}
In genus $3$ the hyperelliptic locus is contained in $\textrm V\wt(2)$. Since each hyperelliptic curve of genus $3$ has $8$ Weierstrass points, the map $\pi$ restricted to it has degree greater than $1$. Since each hyperelliptic Weierstrass point has weight $3$, a local check performed carefully in \cite{Diaz1} shows that the degree of $\pi$ restricted to $\textrm VH_3$ is $16$. On the other hand it is known (see e.~g.~\cite{Diaz3}) that each genus $3$ curve possessing a hyperflex has only one such. So the degree of $\pi$ restricted to $\Hcal$, the hyperflex locus, is $1$ and then for $g=3$ formula \eqref{eq:wt2} can be correctly written as
$$
16\cdot [\overline H_3]+[\Hcal] = \left[\overline{\wt(2)}\right]=452\lambda-48\delta_0-124\delta_1.
$$
The calculation $[\overline H_3]=9\lambda-\delta_0-3\delta_1$ was already reviewed in Section \ref{estev}. Then, the class of the curves possessing a hyperflex is given by
\be
[\Hcal]=308\lambda-32\delta_0-82\delta_1.
\ee

\begin{example}
Consider a pencil of plane quartic curves with smooth generic fibre. Since it has no reducible fibres, the degree of $\delta_1$ is zero on this family. The degree of $\delta_0$ is $27$ while the degree of $\lambda$ is $3$. Then in a pencil of plane quartics one finds precisely 
$308\cdot 3-32\cdot 27= 60$ hyperflexes, as predicted by Proposition \ref{prop:hyperflexes} using the automatic degeneracy formula by Patel and Swaminathan.
\end{example}

\section{Further examples and open questions} \label{sec:Examples}

The purpose of this section is to show how the theory of Weierstrass points on Gorenstein curves may help to interpret some phenomenologies that naturally occur in the geometry and intersection theory of the moduli space of curves.

\subsection{The Examples}\label{sec:examples}

\begin{example}
Let $\pi\colon \TS\sra S\defeq \Spec \C\llbracket t\rrbracket$ be a family of stable curves, such that

\begin{enumerate} 
\item $\TS$ is a smooth surface analytically equivalent to $xy-t=0$,
\item $\TS_\eta$ is a smooth curve of genus $g$, and
\item $\TS_0$ is a stable  uninodal curve,  union of a smooth curve $X$ of genus $g-1$ intersecting transversally an elliptic curve $E$ at a point $A$, that is, $\TS_0=X\cup_A E$.
\end{enumerate}

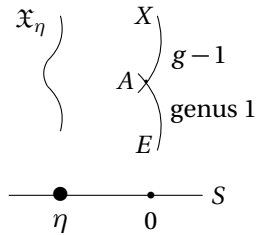
\begin{figure}[ht]
\begin{tikzpicture}[scale=0.85] 
    \draw (-1.5,-0.3) to (1.5,-0.3);    
    \node at (-0.7,-0.3) {\Large{$\bullet$}};
    \node at (0.7,-0.3) {\tiny{$\bullet$}};
    \node at (-1.15,2.4) {\small{$\mathfrak X_\eta$}};
    \draw (-0.7,2.5) to[bend left] (-0.9,1.8);
    \draw (-0.9,1.8) to[bend right] (-0.9,1.4);
    \draw (-0.9,1.4) to[bend left] (-0.7,0.7);
    \node at (-0.7,-0.75) {\small{$\eta$}};
    \node at (0.7,-0.7) {\small{$0$}};
    \draw (0.8,2.5) to[bend left] (0.5,1.3); 
    \draw (0.5,1.6) to[bend left] (0.8,0.4);
    \node at (0.6,2.5) {\small{$X$}};
    \node at (0.6,0.5) {\small{$E$}};
    \node at (1.67,1) {\small{genus $1$}};
    \node at (1.44,1.85) {\small{$g-1$}};
    \node at (1.75,-0.3) {\small{$S$}};
     \node at (0.3,1.5) {\small{$A$}};
    \node at (0.63,1.44) {$\cdot$};
\end{tikzpicture}   
\caption{A family of stable curves degenerating to a general member of $\Delta_1\subset \overline M_g$.}\label{fig1}
    \end{figure}

One says that $P_0\in \TS_0\setminus\{ A\}$ is a {\em limit of a Weierstrass point} if, possibly after a base change, there is a rational section $P\colon S\sra \TS$ such that $P_\eta$ is a Weierstrass point on $\TS_\eta$. The limit of Weierstrass points are very well understood for reducible curves of compact type, by means of many investigations due to Eisenbud, Harris and their school. In fact several classical references (see e.g.~\cite{Diaz2,EH1}) show that

\begin{enumerate}
\smallskip
\item [(a)] if $P_0\in E$, then $P_0\neq A$ is a ramification point of the linear system $\O(gA)$. Applying the Brill--Segre formula (\ref{eq:totwei}), the total weight $\wt_V$ of the ramification points of the linear system $V = H^0(E,\O(gA))$ is $g^2$, including the point $A$. Thus there are at most $g^2-1$ Weierstrass points on the smooth generic fibre degenerating to the elliptic component.  All the ramification points of $V$ are simple, as one can check via the sequence of dimensions 
\[
\dim V\geq \dim V(-A)\geq \cdots\geq \dim V(-gA)\geq \dim V(-(g+1)A)=0.
\]
\smallskip
\item [(b)] If $P_0\in X\setminus \{A\}$ is a limit of a Weierstrass point,  then it is a ramification point  of the linear system $W \defeq H^0(X,K_X(2A))$. Applying the Brill--Segre formula (\ref{eq:totwei}) once more, by replacing $r+1$ by $g$  and $d$ by $2g-2$, one obtains
$$
\wt_W=2g(g-1)+(g-2)g(g-1)=(g-1)(2g+g^2-2g)=g^2(g-1).
$$
The point $A$ contributes with weight $g-1$ (as one easily checks by looking at its vanishing sequence) and thus there are at most $(g-1)^2(g+1)$  Weierstrass points on $\TS_\eta$ degenerating to $X$. 
\end{enumerate}
It follows that no more than 
$$
(\wt_V-1)+(\wt_W-g+1)=\wt_V+\wt_W-g={g^3-g}
$$
Weierstrass points on $\TS_\eta$ can degenerate to $\TS_0$. Since the total weight of the Weierstrass points of $\TS_\eta$ is $g^3-g$, it follows that all the ramification points of the linear systems $V$ and $W$ are indeed limits of Weierstrass points. There are exactly $g^2-1$ distinct Weierstrass points degenerating on $E$ and a total weight of  $(g-1)^2(g+1)$ Weierstrass points on $\TS_\eta$ degenerating on $X$. Moreover, the counting argument shows that the node $A$ is not a limit.
Notice that $g^2-1$ is the weight of a cuspidal curve of arithmetic genus $g$, according to Example \ref{weightcusp24}. This is not a coincidence.

The situation just described is related to the behavior of a family of smooth genus $g$ curves, degenerating to a cuspidal curve of arithmetic genus $g$. The relative dualising sheaf coincides with the canonical sheaf on smooth fibres. The Weierstrass points of the smooth fibres degenerate to the Weierstrass points on the special fibre (with respect to the dualising sheaf), including the cusp, and the cusp has weight $g^2-1$ in the sense of Widland and Lax. Let us now show how to construct a model of the original family contracting the elliptic curve to a cusp.
The idea is to consider 
$\rds(-X)$,  the dualising sheaf twisted by $-X$ (a Cartier divisor, due to the smoothness hypothesis on $\TS$). We have
\[
\pi_*\rds(-X)\otimes\CC(0)\cong H^0(\mathfrak X_0,\rds(-X)|_{\mathfrak X_0}).
\]
Now observe that $h^0(\mathfrak X_0,\rds(-X)|_{{\mathfrak X_0}})\geq g = h^0(X,\omega_X(2A))$. But the restriction map 
\be\label{restriction1}
H^0(\TS_0, \rds(-X)|_{\TS_0})\ra H^0(X, \omega_X(2A)), \qquad \sigma\mapsto \sigma|_X,
\ee
is injective. Indeed, if $\sigma|_{X}=0$ then $\sigma(A)=0$, that is, $\sigma|_{E}\in H^0(\O_E(-A))=0$. Thus $\sigma=0$, which implies that the \eqref{restriction1} is an isomorphism.
Now the sheaf $\mathscr M \defeq \pi_*\rds(-X)$ maps the family $\pi\colon \mathfrak X\ra S$ in
$\PP(\pi_*\rds(-X))$, i.e. we have the following diagram:
\[
\begin{tikzcd}[column sep = large]
\mathfrak X\arrow{r}{\phi_{{\small \mathscr M}}}\arrow[swap]{d}{\pi} & \PP(\pi_*\rds(-X))\arrow{dl} \\
S &
\end{tikzcd}
\]
The generic fibre $\TS_\eta$ is mapped by $\phi_{\mathscr M}$ isomorphically onto its canonical image, a geometrically smooth curve of genus $g$, whereas the special fibre is a cuspidal curve having a cusp in $A$, and the elliptic component of $\TS_0$ is contracted to $A$ by $\phi_{\mathscr M}$. In fact, since the restriction of such a map to $E$ has degree $0$, one has $\phi_{\mathscr M}(Q)=\phi_{\mathscr M}(A)$ for all $Q\in E$. Then there are  $g^2-1$ Weierstrass points degenerating onto the cusp: this number equals the weight of the cusp as a Weierstrass point with respect to the dualising sheaf.
\end{example}

\begin{example}	
As another illustration of the same phenomenology, consider the classical case of a pencil of cubics, for instance
$$
\Ccal_t: \quad zy^2-x^3-tyz^2=0.
$$
The generic fibre $\Ccal_t$ is smooth. It has $9$ flexes, as classically known. But $\Ccal_0$  has only one smooth flex at $F\defeq (0:1:0)$. Thus the remaining flexes collapse to the cusp $P\defeq (1:0:0)$, as is visible by considering the normalisation. 
The Weierstrass points with respect to the linear system of lines can be detected via the  Wronskian determinant by Widland and Lax. It predicts that the cusp has weight $8$. The cubic $\Ccal_0$ is  the image of the map $(x_0^3, x_0x_1^2,x_1^3)\colon \PP^1\sra \PP^2$. In the open affine set $x_0=1$, it is just the map $t\sra (t^2,t^3)$.  
Notice that $\dd t$ is a 
regular differential at $P$ of $\AA^1\subset \PP^1$ and 
then $\sigma \defeq \dd t/t^2$ generates the dualising sheaf at the cusp (where $(t^2)$ is the conductor of $\O_P\subset \widetilde{\O}_P$). One has 
\[
(t^n)'\sigma\defeq \dd(t^n)=nt^{n-1}\dd t=nt^{n+1}\frac{\dd t}{t^2}=nt^{n+1}\sigma
\]
from which $(t^n)'=nt^{n+1}$.
The Wronskian around the point $P$ is then given by
\[
\begin{vmatrix}
1&t^2&t^3\cr 
0&(t^2)'&(t^3)'\cr 
0&(t^2)''&(t^3)''
\end{vmatrix}\,\,=\,\,
\begin{vmatrix}1&t^2&t^3\cr 
0&2t^3&3t^4\cr 
0&6t^4&12t^5
\end{vmatrix}\in t^8\cdot \CC[t].
\]
\end{example} 

\begin{example}
In \cite{EH1}, Eisenbud and Harris study limits of Weierstrass points on a nodal reducible curve $C$ which is the union of a curve $X$ of genus $g-i$ together with $1\leq i\leq g$ elliptic tails, a curve of arithmetic genus $g$. More precisely, if $\TS\ra S$ has smooth generic fibre $\TS_\eta$ and $\TS_0$ is semistably equivalent to $C$, then each elliptic tail carries $g^2-1$ limits of Weierstrass points on nearby smooth curves: these are in turn the ramification points of the linear systems $\O_{E_j}(A_j)$, where $A_j$ is the intersection point $X\cap E_j$. The remaining Weierstrass points of $\TS_\eta$ degenerate on smooth points of $X$. The theory predicts that if $P_0\in X$ is a limit of a Weierstrass point $P_\eta\in \TS_\eta$, then it is a ramification point of a linear system $V\in G(g, H^0(K_X(2A_1+\cdots +2A_i))$ such that $A_i$ is a base point of $V(-A_1-\cdots-A_i)$. If $\widehat{X}$ is the $i$-cuspidal curve got by making each $A_j$ into a cusp, as explained in \cite{serre2}, then $V=\langle \nu^*\omega_1,\ldots, \nu^*\omega_g\rangle$, where $(\omega_1,\ldots,\omega_g)$ is a basis of $H^0(\widehat X,\omega_{\widehat X})$ and $\nu\colon X\sra\widehat{X}$ is the normalisation.  
This linear system coincides with the one induced by the dualising sheaf of the irreducible curve with $i$ cusps that $X$ normalises.

     \begin{figure}[ht]
 \begin{tikzpicture}[scale=0.85]   
    \draw (-1,1.6) to[bend left] (-1,-1.6);
    \draw (-1.3,1) to (1.3,1);
    \draw (-1.3,0.3) to (1.3,0.3);
    \draw (-1.3,-1) to (1.3,-1);
    \node [anchor=west,right] at (1.3,1) {\small{$E_1$}};
    \node [anchor=west,right] at (1.3,0.3) {\small{$E_2$}};
    \node [anchor=west,right] at (1.5,-0.2) {\tiny{$\vdots$}};
    \node [anchor=west,right] at (1.3,-1) {\small{$E_i$}};
    \node [right] at (-1,1.58) {\small{$g-i$}};
    \node [anchor=west,right] at (-1.03,-1.6) {\small{$X$}};
    \node [anchor=west,right] at (3.45,-1.6) {\small{$\widehat X$}};
    \draw (4,1.9) to[bend left] (4,1.2);
    \draw (4,1.2) to[bend left] (4,0.5);
    \draw (4,0.5) to[bend left] (4,-0.2);
    \draw (4,-0.2) to[bend left] (4,-0.9);
    \draw (4,-0.9) to[bend left] (4,-1.6);
    \node [anchor=west,right] at (4.05,1.2) {\small{$A_1$}};
    \node [anchor=west,right] at (4.05,0.5) {\small{$A_2$}};
    \node [anchor=west,right] at (4.15,-0.05) {\small{$\vdots$}};
    \node [anchor=west,right] at (4.05,-0.9) {\small{$A_i$}};
 \end{tikzpicture}
 \caption{Stable reduction of a degeneration to a cuspidal curve.}
    \end{figure}
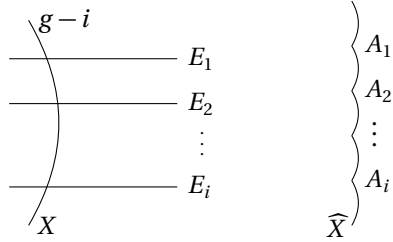
\end{example}

\begin{example}\label{excusp} 
Let $C$ be a smooth complex curve of genus $g-1\geq 1$ and let $\widehat{C}\sra C$ be a family of cuspidal curves parametrised by $C$ itself contructed as follows. If $Q\in C$ is a point, the fibre $\widehat{C}_Q$ is the cuspidal curve obtained from $C$ by creating a cusp at the point $Q$, that is, the cuspidal curve associated to the modulus $2Q$ in the sense of \cite[p.~61]{serre2}. In other words, $\widehat{C}_Q$ is the curve such that $\O_{\widehat{C}_Q,P}=\O_{C,P}$ if $P\neq Q$, whilst $\O_{\widehat{C}_Q,Q}$ is the subring of $\O_{C,Q}$ of the regular functions whose derivatives vanish at $Q$. One wonders which fibres of the family carry special Weierstrass points (with respect to the dualising sheaf) away from the cusp $\{Q\}$. Let $\nu: C\sra \widehat{C}_Q$ be the normalisation of $\widehat{C}_Q$. Then $\nu^*\omega_{\widehat{C}_Q}=K_C(2Q)$ and then the special ramification points, but $Q$, of $\widehat{C}_Q$ are the special ramification points of the linear system $K_C(2Q)$. For general $Q$, one cannot expect to find any  such point. So, solving the problem amounts to finding the locus $\textrm{SW}_1$ of all the pairs $(P,Q)\in C\times C$ such that $P$ is a special ramification point of $K_C(2Q)$.  The number $N(g)$ of such pairs is obtained by putting $i=1$ in \cite[formula (20)]{CuGaEs1}:
$$
N(g)\defeq \int_{C\times C}[\textrm{SW}_1]=6g^4 + 14g^3 + 10g^2 - 14g - 16.
$$
Notice that $N(1)=0$, because a rational cuspidal curve of arithmetic genus $1$ (i.e. a plane cuspidal cubic) has no hyperflexes.
\end{example}

\begin{example} 
Example \ref{excusp} can be interpreted within the geometrical framework of moduli space of stable curves as follows. Let $\Ccal \sra X$ be a family such that $\Ccal_Q$ is the curve $X\cup_{Q\sim 0}E$, where $(E,0)$ is an elliptic curve.
Then  $P_0\in X$ is a limit of a special Weierstrass point if and only if it is a special Weierstrass point of the linear system $K_C(2P)$. This fact has been generalised first of all in \cite{CuGaEs}: if $X\cup_AY$ is a uninodal stable curve of arithmetic genus $g$ union of a smooth curve of genus $i$ and a smooth curve of genus $g-i$ then $P_0\in X$ is limit of a special Weierstrass point on $\TS_\eta$ if and only if either $P_0$ is a ramification point of the linear system $K_X((g_Y+1)A)$ or $P_0$ is a ramification point of the linear system $K_X((g_Y+2)A)$ and $A$ is a Weierstrass point for the component $Y$. In case $Y$ is an elliptic curve, i.e. without Weierstrass points, the limits on $X$ are solely the ramification points of $K_X(2P)$, as claimed.
\end{example}

\begin{example} 
The first example not immediately treated by the theory of Eisenbud and Harris is that of a family $\TS\ra S$ of curves of genus $3$ such that the special fibre $\TS_0$ is the union of two elliptic curves intersecting transversally at two points $A_1$ and $A_2$ (the ``banana curve'').

\begin{figure}[ht]
\begin{tikzpicture}[scale=0.85] 
\draw (2,1) to[bend left] (0,-1);
\node [anchor=west,right] at (-0.7,-1) {$\tiny{E_1}$};
\draw (0.3,-1.2) to[bend left] (2.3,0.8);
\node [anchor=west,right] at (2.3,0.8) {$\tiny{E_2}$};
\end{tikzpicture}\caption{The banana curve: an example of a genus $3$ curve carrying a $1$-parameter family of limits of Weierstrass points.}    
\end{figure}
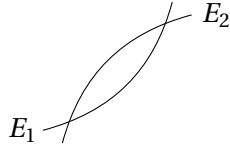

In this case each point on each component can be limit of Weierstrass points, in the sense that for each point $P_0$, say in $E_1$, there exists a smoothing family $\TS\sra S$ such that $P_0$ is limit of a Weierstrass point of a curve of genus $3$. All the Weierstrass points distribute themselves in twelve points on $E_1$ and twelve points on $E_2$. Esteves and Medeiros prove in \cite{Esteves_2002} that the variety of limit canonical system of the ``banana curve'' is parametrised by $\PP^1$. 

Indeed each $P_0\in E_i$ determines uniquely a point in the pencil of linear systems $V\in G(3, H^0(\O(2A_1+2A_2))$ which contains $H^0(\O(A_1+A_2))$. Thus for each component there is a $12:1$ ramified covering $E_i\sra \PP^1$ and the (fixed) ramification points are the limits of special Weierstrass points on nearby smooth curves. Also this example may be interpreted in terms of the theory of Widland and Lax (see \cite{CopGat} for details). In fact the linear system $V_{P_0}$ defined on $E_1$ maps $E_1$ to a plane quartic with a tacnodal singularity ($\delta_A=2$, local analytic equation $(y-x^2)^2=0$) at  the coincident images of $A_1$ and $A_2$ . Then the limits of Weierstrass points on $E_1$ are precisely the smooth flexes, while the information about the  Weierstrass points degenerating on the other components is lost in the tacnode. Notice that according the theory of Widland and Lax a tacnode must have weight at least $\delta\cdot 3\cdot 2=2\cdot 3\cdot 2=12$.
\end{example}

\subsection{Open Questions}

\subsubsection{Porteous Formula with excess} \label{porexc}
Consider the loci
\begin{align*}
\wt(2)&\defeq \Set{[C]\in M_g | C \textrm{ has a special Weierstrass point}},\\
\mathbb D_{g-1}&\defeq \Set{[C]\in M_g | C \textrm{ has a special Weierstrass point of type }g-1},\\
\mathbb D_{g+1}&\defeq \Set{[C]\in M_g | C \textrm{ has a special Weierstrass point of type }g+1}.
\end{align*}
Although  $\wt(2)$ is clearly equal to the set-theoretic union $\mathbb D_{g-1}\cup \mathbb D_{g+1}$, it is not obvious that 
\[
\left[\overline{\wt(2)}\right]=\left[\overline{\mathbb D}_{g-1}\right]+\left[\overline{\mathbb D}_{g+1}\right].
\]
This is the main result of \cite{Gatto1}. Within the general framework discussed in Section~\ref{sec:specialWP}, consider the maps of vector bundles
\[
\begin{tikzcd}
\pi^\ast\EE_\pi \arrow{rr}{\mathsf D^{g-2}}\arrow[dash]{dr} & &
J_\pi^{g-2}\omega_\pi \arrow[dash]{dl} \\
& \Ccal \arrow{d}{} & \\
& T &
\end{tikzcd} \qquad \begin{tikzcd}
\pi^\ast\EE_\pi \arrow{rr}{\mathsf D^{g}}\arrow[dash]{dr} & &
J_\pi^{g}\omega_\pi \arrow[dash]{dl} \\
& \Ccal \arrow{d}{} & \\
& T &
\end{tikzcd}
\]
The loci $\overline{\mathbb D}_{g-1}$ and $\overline{\mathbb D}_{g+1}$ are in fact in the degeneracy loci of the above maps; however these maps degenerate identically  along the special singular fibre which are divisors of $\Ccal$. So, to compute the class of the loci of $\overline{\mathbb D}_{g-1}$ and $\overline{\mathbb D}_{g+1}$ one should dispose of a Porteous formula with excess, generalising the residual formula for top Chern classes as in \cite[Example 14.1.4]{Ful}.
To our knowledge, such formulas are not known up to now.

\subsubsection{Computing automatic degeneracies}
It is an interesting problem, already raised in \cite{InvParts}, to compute the function $\AD^m(f)$ of automatic degeneracies (as discussed in Section~\ref{autdeg}) for more complicated plane curve singularities than the node. Some results for low values of $m$ have already been obtained in \emph{loc.~cit}. For instance it would be very useful to be able to determine the function $\AD(f)$ for cusps, ordinary triple points, tacnodes.

\subsubsection{Porteous formula for Coherent sheaves} 
To study situations like \ref{porexc} but avoiding the locally free replacement of the principal parts, S. Diaz proposed in \cite{Diaz3} a Porteous formula for maps of coherent sheaves. This was a question asked by Harris and Morrison in \cite{ModCurves1}. The purpose is that of getting rid of two issues at once: excess contributions, and the lack of local freeness of principal parts of the dualising sheaf at singularities. Diaz's theory is nice and elegant. However the main example he proposes is the computation of the hyperelliptic locus in genus $3$, which Esteves computed as sketched in Section \ref{estev}, again using  locally free substitute of principal parts. It would be interesting to work out more examples  to extract all the potential of  Diaz' extension of  Porteous' formula for coherent sheaves.

\subsubsection{Dimension estimates} 
Recall the definition \eqref{eq:wtk} of $\wt(k)$.
In \cite{GattoPonza} it is proven that for $g\geq 4$ the locus $\wt(3)$ of curves possessing a special Weierstrass point of weight at least $3$ has the expected codimension $2$. It is a hard problem to determine the irreducible components of $\wt(k)$ and their dimensions. For instance Eisenbud and Harris prove that if $k\leq [g/2]$ then $\wt(k)$ has at least  one irreducible component of the expected codimension $k$. In  general, however, the problem is widely open. It would be natural to conjecture that $\wt(k)\subset M_g$ has the expected codimension $k$ if $g\gg 0$, but there is really no rigorous evidence to support such a guess.

\subsubsection{Computing new classes} 
Only a handful of classes of  geometrically defined loci of higher codimension in $\overline{M}_g$ have been computed. For instance Faber and Pandharipande have determined the class of the hyperelliptic locus in $\overline{M}_4$ via stable maps \cite{Faber_2005}. Let $\Ccal\sra S$ be a family of stable curves of genus $g\geq 5$ parametrised by a smooth complete  surface $S$. Then many singular fibres belonging to boundary strata of $\overline{M}_g$ of higher codimension can occur. If $\pi\colon\TS\sra S$ is a family of stable curves of genus $4$ parameterized by a complete scheme of dimension at least $2$, then Faber and Pandharipande are able to compute the locus of points in $S$ corresponding to hyperelliptic fibres. Esteves and Abreu (private communication) are able to compute the  class $[\overline{H}_4]$ using the same method we discussed in Section \ref{estev}.
However  it seems a hard problem to determine the class in $A_{3g-5}(\overline{M}_g)$ (already for $g=4$) of the locus $\overline{\wt(3)}$. This would be the push forward of the third Chern class of
\[
J^2_\pi\left(\omega_\pi^{g(g+1)/2}\otimes \bw^g\EE_\pi^\vee\right),
\]
where $J^2_\pi$ is the locally free replacement constructed in the previous sections. Unfortunately, one has no control on the degree of the restriction of $\pi$ to the irreducible components of $\overline{\textrm V\wt(3)}$. In genus $4$ this locus should contain, with some multiplicity, the hyperelliptic locus, the (nonempty) locus of curves possessing a Weierstrass point with gap sequence $(1,2,3,7)$ and the (nonempty) locus of curves possessing a Weierstrass point with gap sequence $(1,2,4,7)$. These loci all have the expected codimension $2$ (by \cite{Laxuniversal}), but as far as we know their multiplicities in $\overline{\wt(3)}$ are not known.

\clearpage
\bibliographystyle{amsplain}
\bibliography{bib}

\end{document}